\documentclass[a4paper,11pt,english]{smfart}

\usepackage{amssymb,mathtools}
\usepackage{url}
\usepackage[T1]{fontenc}
\RequirePackage{calrsfs}
\DeclareSymbolFont{rsfscript}{OMS}{rsfs}{m}{b}
\DeclareSymbolFontAlphabet{\mathrsfs}{rsfscript}
\usepackage[utopia,expert]{mathdesign}
\usepackage{lscape}
\usepackage{array, boldline, makecell, booktabs}

\usepackage{enumitem}

\usepackage{todonotes}

\usepackage[leftbars]{changebar}

\usepackage{palatino}
\usepackage{rotating}
\usepackage{graphicx}

\usepackage{tikz-cd}
\def\proofname{D\'emonstration}

\usepackage[all]{xy}

\makeindex

\def\indexname{Index des notations}

\newtheorem{theo}{Theorem}[section]
\def\thetheo{\thesection.\arabic{theo}}
\newtheorem{prop}[theo]{Proposition}

\newtheorem{lem}[theo]{Lemma}

\newtheorem{coro}[theo]{Corollary}

\renewcommand\theequation{\thetheo}
\def\equat{\refstepcounter{theo}\begin{equation}}
\def\endequat{\end{equation}}

\renewcommand\thetable{\Roman{table}}
\renewcommand\thefigure{\Roman{figure}}

\renewcommand\thesection{\arabic{section}}
\renewcommand\thesubsection{\thesection.\Alph{subsection}}

\def\contentsname{Table des mati\`eres}
\def\listfigurename{Liste des figures}
\def\listtablename{Liste des tables}
\def\appendixname{Appendice}

\def\refname{R\'ef\'erences}

  \def\aG{{\mathfrak a}}  
    
    \def\CM{{\mathbb{C}}}

  \def\gG{{\mathfrak g}}

  \def\lG{{\mathfrak l}}  
\def\MG{{\mathfrak M}}  \def\mG{{\mathfrak m}}  
    \def\NM{{\mathbb{N}}}
  \def\oG{{\mathfrak o}}  
  \def\pG{{\mathfrak p}}

\def\SG{{\mathfrak S}}  \def\sG{{\mathfrak s}}

    \def\ZM{{\mathbb{Z}}}

\def\Gb{{\mathbf G}}

\def\Lb{{\mathbf L}}    
    
    \def\NC{{\mathcal{N}}}
\def\Ob{{\mathbf O}}    \def\OC{{\mathcal{O}}}
  \def\pb{{\mathbf p}}  
    \def\QC{{\mathcal{Q}}}
    
\def\Sb{{\mathbf S}}    \def\SC{{\mathcal{S}}}
  \def\tb{{\mathbf t}}  
    \def\UC{{\mathcal{U}}}
    \def\VC{{\mathcal{V}}}
    
    \def\XC{{\mathcal{X}}}
    \def\YC{{\mathcal{Y}}}
    \def\ZC{{\mathcal{Z}}}

\def\Hrm{{\mathrm{H}}}

\def\Prm{{\mathrm{P}}}

\def\Trm{{\mathrm{T}}}

    \def\NCt{{\tilde{\mathcal{N}}}}

    \def\QCt{{\tilde{\mathcal{Q}}}}

  \def\xti{{\tilde{x}}}  \def\XCt{{\tilde{\mathcal{X}}}}

\def\Bba{{\bar{B}}}          \def\bba{{\bar{b}}}

\def\Eba{{\bar{E}}}          \def\eba{{\bar{e}}}
          \def\fba{{\bar{f}}}

\def\a{\alpha}
\def\b{\beta}
\def\g{\gamma}
\def\G{\Gamma}
\def\d{\delta}
\def\D{\Delta}
\def\e{\varepsilon}
\def\ph{\varphi}
\def\l{\lambda}
\def\L{\Lambda}

\def\o{\omega}
\def\O{\Omega}

\def\th{\theta}

\def\t{\tau}

\def\z{\zeta}

\def\mub{{\boldsymbol{\mu}}}

\DeclareMathOperator{\diag}{{\mathrm{diag}}}

\DeclareMathOperator{\Hom}{{\mathrm{Hom}}}

\DeclareMathOperator{\Mat}{{\mathrm{Mat}}}

\DeclareMathOperator{\Spec}{{\mathrm{Spec}}}

\def\to{\rightarrow}
\def\longto{\longrightarrow}
\def\injto{\hookrightarrow}

\def\fonction#1#2#3#4#5{\begin{array}{rccc}
{#1} : & {#2} & \longto & {#3}  \\
& {#4} & \longmapsto & {#5} 
\end{array}}

\def\DS{\displaystyle}

\def\finl{~$\blacksquare$}

\def\lexp#1#2{\kern\scriptspace\vphantom{#2}^{#1}\kern-\scriptspace#2}
\def\le{\hspace{0.1em}\mathop{\leqslant}\nolimits\hspace{0.1em}}
\def\ge{\hspace{0.1em}\mathop{\geqslant}\nolimits\hspace{0.1em}}

\mathchardef\inferieur="321E
\mathchardef\superieur="321F

\def\eqna{\begin{eqnarray*}}
\def\endeqna{\end{eqnarray*}}

\def\mini{{\mathrm{min}}}
\def\itemth#1{\item[${\mathrm{(#1)}}$]}

\long\def\@car#1#2\@nil{#1}
\long\def\@first#1#2{#1}
\long\def\@second#1#2{#2}
\long\def\ifempty#1{\expandafter\ifx\@car#1@\@nil @\@empty
  \expandafter\@first\else\expandafter\@second\fi}

\theoremstyle{remark}
\newtheorem{rema}[theo]{Remark}

\newtheorem{exemple}[theo]{Example}

\theoremstyle{plain}

\def\BIL{LR}
\def\GAUCHE{L}
\def\CAR{CAR}
\def\FAM{FAM}

\def\xyinj{\ar@{^{(}->}}
\def\xysur{\ar@{->>}}
\def\xymap{\ar@{|->}}

\bigskip

\makeatletter
\def\hlinewd#1{%
\noalign{\ifnum0=`}\fi\hrule \@height #1 %
\futurelet\reserved@a\@xhline}
\makeatother

\newlength\epaisLigne

\newcommand{\longiso}{\stackrel{\sim}{\longrightarrow}}

\makeatletter
\def\hlinewd#1{%
\noalign{\ifnum0=`}\fi\hrule \@height #1 %
\futurelet\reserved@a\@xhline}
\makeatother

\def\Sp{\operatorname{\Sb\pb}\nolimits}
\def\sing{{\mathrm{sing}}}

\def\Sym{{\mathrm{Sym}}}

\def\Qu{{\mathsf{Q}}}
\def\bv{{\mathbf{v}}}

\def\GO{\operatorname{\Ob}\nolimits}

\def\Sp{\operatorname{\Sb\pb}\nolimits}

\begin{document}

\title{A new family of isolated symplectic singularities \\
with trivial local fundamental group}

\author{{\sc Gwyn Bellamy}}

\address{The Mathematics and Statistics Building, University of Glasgow, 
University Place, Glasgow G12 8QQ, Scotland} 
\email{gwyn.bellamy@glasgow.ac.uk}

\pagestyle{myheadings}

\makeatother

\author{{\sc C\'edric Bonnaf\'e}}
\address{IMAG, Universit\'e de Montpellier, CNRS, Montpellier, France} 
\email{cedric.bonnafe@umontpellier.fr}

\author{{\sc Baohua Fu}}
\address{AMSS, HLM and MCM, Chinese Academy of Sciences, 55 ZhongGuanCun East Road, Beijing, 100190, China and School of Mathematical Sciences, University of Chinese Academy of Sciences, Beijing, China}
\email{bhfu@math.ac.cn}

\author{{\sc Daniel Juteau}}
\address{LAMFA, Universit\'e de Picardie Jules Verne, CNRS, Amiens, France}
\email{daniel.juteau@u-picardie.fr}

\author{{\sc Paul Levy}}
\address{Department of Mathematics and Statistics Fylde College,
Lancaster University, Lancaster, LA1 4YF, United Kingdom}
\email{p.d.levy@lancaster.ac.uk}

\author{{\sc Eric Sommers}}
\address{Department of Mathematics and Statistics, University of Massachusetts, Amherst,
MA 01003-4515, USA}
\email{esommers@math.umass.edu}
\date{\today}

\thanks{
C.~B. is partially supported by the ANR projects GeRepMod (ANR-16-CE40-0010-01) and CATORE (ANR-18-CE40-0024-02).
B.~F. is supported by the NSFC grant No. 11688101.
D.~J. is grateful for the support of the ANR project GeRepMod (ANR-16-CE40-0010-01) and of the Charles Simonyi Endowment at the Institute for Advanced Study.
}

\begin{abstract}
We construct a new infinite family of 4-dimensional isolated symplectic singularities  with trivial 
local fundamental group, answering a question of Beauville raised in 2000.
Three constructions 
are presented for this family: (1) as singularities in blowups of 
the quotient of $\mathbb{C}^4$ by the dihedral group of order $2d$, (2)  
as singular points of Calogero-Moser spaces 
associated with dihedral groups of order $2d$ at equal parameters, 
(3) as singularities of a certain Slodowy slice in the $d$-fold 
cover of the nilpotent cone in $\sG\lG_d$.
\end{abstract}

\maketitle

\markboth{\sc Bellamy-Bonnaf\'e-Fu-Juteau-Levy-Sommers}{\sc Symplectic singularities with trivial 
fundamental group}

\section{Introduction}


Symplectic singularities were introduced by Beauville~\cite{beauville}. Their introduction has led to numerous important developments in both algebraic geometry and geometric representation theory. Basic examples of symplectic singularities 
include symplectic quotient singularities and singularities in normalizations of nilpotent orbit 
closures in semisimple Lie algebras~\cite[Section 2]{beauville}.  Note that any two dimensional 
symplectic singularity is  just a rational double point.

An isolated symplectic singularity of dimension $\ge 4$ is a normal isolated singularity whose 
smooth locus admits a holomorphic symplectic 2-form~\cite[$(1.2)$]{beauville}. 
It follows that such a singularity is canonical Gorenstein, hence its local 
fundamental group is finite~\cite{Br}. Symplectic quotient singularities give many examples of isolated 
symplectic singularities, but they all have  non-trivial local fundamental groups. 

Minimal nilpotent orbit closures $\overline{\OC}_\mini^\gG$ in simple Lie 
algebras $\gG$ give examples of isolated symplectic singularities with trivial local 
fundamental group.  In this case, the projective tangent cone $\Prm\Trm_0(\overline{\OC}_\mini^\gG)$ of 
$\overline{\OC}_\mini^\gG$ at $0$ is isomorphic to $\Prm \OC_\mini^\gG = \OC_\mini^\gG/\CM^\times$, which 
is smooth and  $G$-homogeneous (where $G$ is the adjoint Lie group of $\gG$). It turns out that 
the smoothness of projective tangent cones characterizes $\overline{\OC}_\mini^\gG$, by the following result. 

\medskip

\begin{theo}[Beauville]\label{theo:beauville}
Let $(\XC,x)$ be an isolated symplectic singularity whose projective 
tangent cone at $x$ is smooth. Then there exists a simple Lie algebra 
$\gG$ such that $(\XC,x)$ is locally analytically isomorphic 
to $(\overline{\OC}_\mini^\gG,0)$. 
Moreover, the singularity $(\overline{\OC}_\mini^\gG,0)$ has trivial local fundamental group if and 
only if $\gG \not\simeq \sG\pG_{2n}(\CM)$ for any $n \ge 1$. 
\end{theo}

\medskip

If $(\XC,x)$ is an isolated symplectic singularity whose projective 
tangent cone at $x$ is smooth, then the simple Lie algebra $\gG$ such that 
$(\XC,x)$ is locally analytically isomorphic to $(\overline{\OC}_\mini^\gG,0)$ is uniquely determined.  
It can be recovered as the Lie algebra structure on the cotangent space 
$\Trm_x^*(\XC)$ of $\XC$ at $x$ induced by the Poisson bracket. In other words, 
\equat\label{eq:poisson-lie}
\Trm_0^*(\overline{\OC}_\mini^\gG) \simeq \gG
\endequat
as Lie algebras.


When $\gG \simeq \sG\pG_{2n}(\CM)$,  the singularity $(\overline{\OC}_\mini^\gG,0)$ is locally analytically isomorphic to 
$(\CM^{2n}/\mub_2,0)$, where $\mub_2=\{1,-1\}$ acts by multiplication on $\CM^{2n}$, 
so the local fundamental group is isomorphic 
to $\mub_2$. 

In the same paper, Beauville asked~\cite[$($4.3$)$]{beauville} 
whether there exist other isolated symplectic singularities with trivial local
fundamental group. In the intervening two decades, no other examples have come to light.

Constructing examples of 
new isolated symplectic singularities is also motivated by the long-standing 
conjecture (attributed to LeBrun and Salamon) that any Fano contact manifold is isomorphic 
to $\Prm \OC_\mini^\gG$ for some simple Lie algebra $\gG$.
In fact, if $Z$ is a Fano contact manifold and $L$ is the contact
line bundle, then there exists a map $L^* \to X$ which contracts
the zero section to one point $o$. The contact structure on $Z$
induces a symplectic structure on the complement of the zero
section in $L^*$, which gives a symplectic form on $X \setminus
\{o\}$. This implies that $(X,o)$ is an isolated symplectic
singularity. By the same argument as in \cite[Proposition 4.2]{beauville}, 
$X$ has trivial local fundamental group except in the case where $Z \simeq \mathbb{P}^{2n-1}.$

The aim of this paper is to describe a new infinite family of 4-dimensional isolated 
symplectic singularities with trivial local
fundamental group. We will give three different constructions of this family: 
one from blowups of  symplectic quotient singularities, 
one from Calogero-Moser spaces associated with dihedral groups, and the last one from slices 
in the $d$-fold cover of the nilpotent cone of $\mathfrak{sl}_d$.

For this,  
let $V$ be a complex vector space of dimension $2$, let $d \ge 4$, 
let $W_d \subset \Gb\Lb_\CM(V)$ denote the dihedral group of order $2d$ and let 
$\QC(d)$ denote the symplectic quotient singularity $(V \times V^*)/W_d$. 
We define two varieties:
\begin{itemize}
\item[$\bullet$] Let $\QCt(d)$ denote the blowup of $\QC(d)$ at its singular locus.

\item[$\bullet$] Let $\ZC(d)$ denote the Calogero-Moser space associated 
with $W_d$ {\it at non-zero equal parameters}~\cite{bonnafe diedral}: 
it is a Poisson deformation of $\QC(d)$.
\end{itemize}
Then both $\QCt(d)$ and $\ZC(d)$ are normal varieties of dimension $4$, which have  
a unique singular point (denoted by $0$ for in both cases). Now $(\ZC(d),0)$ is 
a symplectic singularity by general results about Calogero-Moser spaces~\cite{gordon icra}. 
It is of course an isolated singularity. As will be shown in the paper, $(\QCt(d),0)$ 
is also a symplectic singularity (Corollary \ref{coro:symplectic}), which is locally analytically isomorphic to $(\ZC(d),0)$ (Corollary \ref{coro:blowup}). In the case  
$\QCt(5)$, this was observed in~\cite[Remark 12.8]{FJLS}. 
Our first main result is the following:


\begin{theo}\label{theo:main 1}
The following statements hold:
\begin{itemize}
\itemth{a} The symplectic singularity $(\ZC(d),0)$ has trivial 
local fundamental group.

\itemth{b}  The symplectic singularity $(\ZC(4),0)$ 
is locally analytically isomorphic to $(\overline{\OC}_\mini^{\sG\lG_3},0)$. 

\itemth{c} 
If $d \ge 5$, then the singularity $(\ZC(d),0)$ is not locally analytically isomorphic  
to $(\overline{\OC}_\mini^\gG,0)$ for any simple Lie algebra $\gG$.

\itemth{d} If $d' > d \ge 4$, then the singularities 
$(\ZC(d),0)$ and $(\ZC(d'),0)$ are not locally analytically isomorphic.
\end{itemize}
\end{theo}


Of course, the analogous statements also hold for the singularity $(\QCt(d),0)$. 
The proof relies on the explicit computation of the equations 
for $\ZC(d)$ (see~\cite{bonnafe diedral egal}) and $\QCt(d)$. 
A key role is played by an action of $\Sb\Lb_2$ on both $\ZC(d)$ 
and $\QCt(d)$. 

Symplectic singularities of Calogero-Moser spaces associated with the 
complex reflection groups denoted by $G(d,1,n)$ in Shephard-Todd 
classification~\cite{ST} are relatively well-understood, as they 
are quiver varieties~\cite{EG}. However, we think it might be interesting to study the 
symplectic singularities of Calogero-Moser spaces associated 
with other complex reflection groups. For the infinite family $G(de,e,n)$ 
(note that $W_d=G(d,d,2)$), and the $34$ exceptional groups, it is 
appealing to expect to find more interesting examples of symplectic singularities. 

The third construction is a certain cover of a type $A$ Slodowy slice~\cite{slodowy}. 
If $\mu$ is a partition of $d$, let $\OC_\mu$ denote the associated nilpotent 
orbit in $\sG\lG_d(\CM)$. Let $\NC_d$ denote the nilpotent cone of $\sG\lG_d(\CM)$. 
We denote by $\SC_{d-2,2}$ a Slodowy slice associated with $\OC_{d-2,2}$ and 
set $\XC(d) = \SC_{d-2,2} \cap \NC_d$. The nilpotent cone $\NC_d$ admits a 
$\mub_d$-covering $\pi_d : \NCt_d \longto \NC_d$
which is unramified above the regular nilpotent orbit $\OC_d$ and bijective 
above the branch locus.
Define $\XCt(d)=\pi_d^{-1}(\XC(d))$, denote by $x_d$ the unique element of $\OC_{d-2,2} \cap \SC_{d-2,2}$ 
and let $\xti_d$ denote its unique preimage in $\XCt(d)$. Then $\xti_d$ is the unique singular point 
of $\XCt(d)$ and:


\begin{theo}\label{theo:main 2}
If $d \ge 4$, the symplectic singularities $(\ZC(d),0)$ and 
$(\XCt(d),\xti_d)$ are locally analytically isomorphic.
\end{theo}


The proof of this theorem is rather indirect. The cyclic group $\mub_d$ acts on the singularity 
$(\ZC(d),0)$ and the quotient can be identified with the Calogero-Moser space (at a non-generic parameter) 
$(\ZC(d,1,2),x_0)$ associated to $G(d,1,2)$.  By~\cite{EG,MarsdenWeinsteinStratification}, 
the latter space is isomorphic to a certain quiver variety $(\MG_{\lambda}(\bv),x_0)$ 
associated to the framed affine quiver of type $A$, which can be shown to be isomorphic 
to $\XC(d)$ by using results of Nakajima~\cite{Nak1994} (see also~\cite{Ma}). 
This implies that both $\ZC(d)$ and $\XCt(d)$ are $d$-fold coverings of $\XC(d)$, 
which are shown to be locally holomorphically 
isomorphic by using the triviality of the local fundamental group of $\ZC(d)$.

It is interesting to remark that the singularity $(\XCt(5),\xti_5)$ is exactly the generic 
singularity of the nilpotent orbit closure of $A_4+A_3$ along the codimension 4 boundary 
$A_4+A_2+A_1$ in $E_8$.  More surprisingly, the whole of $\NCt_5$ is in fact isomorphic to 
the Slodowy slice of $(A_4+A_3, A_4)$ in $E_8$  (see~\cite[Section 12.3]{FJLS}).


\bigskip

\noindent{\bf Acknowledgements.} We wish to thank S. Baseilhac, A. Brugui\`eres and C. Xu for useful discussions about fundamental groups.
We also thank the group of A.~Hanany, and particularly A.~Bourget, for many conversations about physics and symplectic singularities; based on the Hilbert series in \S \ref{subsec:hilb}, Z.~Zhong proposed a magnetic quiver whose Coulomb branch should provide yet another description of our singularities.

\bigskip

\noindent{\bf Convention, notation.} 
We work over the field $\CM$ of complex numbers (polynomials, Lie algebras, algebraic varieties, schemes 
are supposed to be complex). If $\XC$ is an affine scheme, 
we denote by $\CM[\XC]$ its ring of regular functions. 

We fix a natural number $d \ge 4$, we denote by $\mub_d$ the cyclic group of order $d$ 
and we fix a primitive $d$-th root of unity $\z$. Let $V=\CM^2$ and denote by $(x,y)$ 
its canonical basis and by $(X,Y)$ its dual basis. Thus, $X$, $Y \in \CM[V]$ and we may view 
$x$ and $y$ as elements of $\CM[V^*]$. If $j \in \ZM$ or $\ZM/d\ZM$, we set
$$s_j=\begin{pmatrix} 0 & \z^j \\ \z^{-j} & 0 \end{pmatrix} \in \Gb\Lb_2(\CM)=\Gb\Lb_\CM(V),$$
$s=s_0$, $t=s_1$ and $W_d=\langle s,t \rangle=\langle s_0,s_1,\dots,s_{d-1} \rangle$: 
it is the dihedral group of order $2d$. 

Let $U_2$ denote the standard representation of $\sG\lG_2$ so that  
$\Sym^d(U_2)$ is the unique simple $\sG\lG_2(\CM)$-module of dimension $d+1$. We denote by $\gG_d$ the Lie algebra $\sG\lG_2 \oplus \Sym^d(U_2)$,  
where the Lie structure is determined by the fact that $\Sym^d(U_2)$ is a commutative ideal of 
$\gG_d$ on which $\sG\lG_2$ acts through its natural action.


\section{Deformation and blowup of a symplectic quotient singularity}


\subsection{Symplectic quotient ${\boldsymbol{(V \times V^*)/W_d}}$}
The induced action of $W_d$ on $V \times V^*$ is symplectic, making  the quotient 
$\QC(d)=(V \times V^*)/W_d$ an affine variety with only symplectic singularities.
The singular locus $S \subset (V \times V^*)/W_d$ consists of $W_d$-orbits with non-trivial 
stabilizer, namely the image of the points of the form $(\l(x-\zeta^j y),\mu(Y - \zeta^j X))$, 
with $\l$, $\mu \in \CM$ and $j \in \ZM/d\ZM$. It is easy to check that $S$ is of codimension 2 and 
the singularity of $(V \times V^*)/W_d$ along a nonzero point of $S$ is an $A_1$-singularity.

We denote by 
$(\Psi_k)_{k \ge 0}$ the sequence of polynomials in three indeterminates $q$, $Q$, $e$, defined by 
$$
\begin{cases}
\Psi_0=1,\\
\Psi_1=e,\\
\Psi_k = e \Psi_{k-1} - qQ\Psi_{k-2}, & \text{if $k \ge 2$.}
\end{cases}
$$
(see~\cite[$($1.4$)$]{bonnafe diedral egal}). Note that $\Psi_k$ is homogeneous 
of degree $k$. Alternatively, $\Psi_k$ can be defined by the generating series
$$\sum_{k \ge 0} \Psi_k \tb^k = \cfrac{1}{1-e\tb+qQ\tb^2}$$
in $\CM[q,Q,e][[\tb]]$. 
In particular, 
\equat\label{eq:psi-e}
\Psi_k(q,0,e)=\Psi_k(0,Q,e)=e^k\qquad\text{and}\qquad 
\Psi_k(q,Q,0)=
\begin{cases}
	0 & \text{if $k$ is odd,} \\
	(-qQ)^{k/2} & \text{if $k$ is even.}
\end{cases}
\endequat

Consider the following $W_d$-invariant polynomials in $\mathbb{C}[V \times V^*]$.
$$
q=xy,  Q=XY, e=xX+yY, a_i=x^{d-i}Y^i+y^{d-i}X^i, 0 \leq i \leq d. $$
By \cite[Theorem 1.6]{bonnafe diedral egal}, we have
\refstepcounter{theo}
\begin{multline}\label{eq:Q(d)}
\QC(d)=\{(q,Q,e,a_0,a_1,\dots,a_d) \in \CM^{d+4}~|~\\
\forall~1\le j \le k \le d-1,~
\begin{cases}
e a_j = q a_{j+1} + Q a_{j-1},\\
a_{j-1} a_{k+1} - a_j a_k  =
(e^2-4qQ) q^{d-k-1} Q^{j-1} \Psi_{k-j}(q,Q,e)
\end{cases} \}.
\end{multline}


\subsection{The variety ${\boldsymbol{\ZC(d)}}$}
The Calogero-Moser spaces associated with complex reflection groups are extensively studied 
in~\cite{EG}. Using the notation of~\cite{bonnafe diedral egal}, we denote by $\ZC(d)$ 
the Calogero-Moser space associated with the dihedral group $W_d$ of order $2d$ 
{\it at equal parameter} $a =1$. By~\cite{EG}, the variety $\ZC(d)$ is endowed with a Poisson structure. The case $a \neq 0$ is isomorphic to the case $a=1$ and 
the only change is that the Poisson structure is multiplied by a scalar: this is irrelevant for our purpose.  
It turns out this Calogero-Moser space is a Poisson deformation of the quotient $\QC(d)$, 
with  equations given as follows~\cite[Theorem~2.9]{bonnafe diedral egal}:
\begin{multline*}
\ZC(d)=\{(q,Q,e,a_0,a_1,\dots,a_d) \in \CM^{d+4}~|~\\
\forall~1\le j \le k \le d-1,~
\begin{cases}
e a_j = q a_{j+1} + Q a_{j-1},\\
a_{j-1} a_{k+1} - a_j a_k  = 
(e^2-4qQ-d^2) q^{d-k-1} Q^{j-1} \Psi_{k-j}(q,Q,e)
\end{cases} \}.
\end{multline*}
Note that the point $0=(0,0,\dots,0) \in \CM^{d+4}$ belongs to $\ZC(d)$ and we denote 
by $\Trm_0^*(\ZC(d))$ the cotangent space of $\ZC(d)$ at $0$.

Let us first recall some facts about the variety $\ZC(d)$ with $d\ge 4$ 
(recall that the Lie algebra $\gG_d=\sG\lG_2 \oplus \Sym^d(U_2)$ has been defined at 
the end of the introduction).


\begin{theo}\label{theo:zd}
The affine variety $\ZC(d)$ is irreducible, normal, of dimension $4$. Moreover:
\begin{itemize}
\itemth{a} The point $0$ is an isolated singularity of $\ZC(d)$ and its associated maximal ideal 
of $\CM[\ZC(d)]$ is a Poisson ideal. This endows $\Trm_0^*(\ZC(d))$ with a Lie algebra structure. 

\itemth{b} We have an isomorphism of Lie algebras 
$$\Trm_0^*(\ZC(d)) \simeq
\begin{cases}
\sG\lG_3(\CM) & \text{if $d=4$,}\\
\gG_d & \text{if $d \ge 5$,}
\end{cases}
$$

\itemth{c} If $d \ge 4$, then $\ZC(d)$ is a symplectic singularity 
in the sense of Beauville.

\itemth{d} The symplectic singularity $(\ZC(4),0)$ is locally analytically isomorphic 
to $(\overline{\OC}_\mini^{\sG\lG_3(\CM)},0)$. 

\itemth{e} If $d \ge 5$, then the symplectic singularity $(\ZC(d),0)$ 
is not locally analytically isomorphic to any $(\overline{\OC}_\mini^\gG,0)$ for a simple 
Lie algebra $\gG$. 

\itemth{f} If $d' > d \ge 4$, then the singularities $(\ZC(d),0)$ and $(\ZC(d'),0)$ 
are not locally analytically isomorphic. 
\end{itemize}
\end{theo}


\begin{proof}
The fact that $\ZC(d)$ is irreducible, normal, of dimension $4$ is a general fact about 
Calogero-Moser spaces~\cite{EG}.

\smallskip

(a) was proved in~\cite{bellamy these} (see also~\cite{bonnafe diedral}).

\smallskip

(b) is proved in~\cite[Proposition~8.4]{bonnafe diedral} and~\cite[Proposition~2.12]{bonnafe diedral egal}.

\smallskip

(c) The fact that $\ZC(d)$ is a symplectic singularity is again a general fact about 
Calogero-Moser spaces~\cite[Proposition~4.5]{gordon icra}.

\smallskip

(d) A computation with {\sc Magma}~\cite{magma} (and the {\sc Champ} 
package built by Thiel~\cite{thiel}) based on the equations of $\ZC(4)$ 
given in~\cite[Proposition~8.3]{bonnafe diedral} 
shows that the projective tangent cone $\Prm\Trm_0(\ZC(4))$ of $\ZC(4)$ at $0$ 
is smooth. So the result follows from (b) and Beauville's 
Theorem~\ref{theo:beauville}~\cite[Introduction]{beauville}. 
We will provide in Remark~\ref{rem:d=4} a proof which does not rely on {\sc Magma} computations.

\smallskip

(e) Assume that $d \ge 5$. By~(b), 
$\Trm_0^*(\ZC(d)) \simeq \gG_d$ is not a simple Lie algebra. 
So, by~\ref{eq:poisson-lie}, 
the singularity $(\ZC(d),0)$ is not locally analytically isomorphic 
to $(\overline{\OC}_\mini^\gG,0)$ for any simple 
Lie algebra $\gG$.

\smallskip

(f) follows from~(b) and the fact that $\dim(\gG_d)=d+4 < d'+4=\dim(\gG_{d'})$.
\end{proof}


\subsection{Blowup}\label{sub:blowup}
A polynomial $f \in \CM[V \times V^*]$ is called {\it $W_d$-semi-invariant} 
if $w(f)=\det(w)f$ for all $w \in W_d$. We denote by $\CM[V \times V^*]^{W_d-\text{sem}}$ 
the set of $W_d$-semi-invariant polynomials: it is a $\CM[V \times V^*]^{W_d}$-module. 
Let $\d=xX-yY$ and $\b_j=x^{d-j}Y^j-y^{d-j}X^j$ for $0 \le j \le d$. Then 
$\d$, $\b_0$, $\b_1$,\dots, $\b_d$ are $W_d$-semi-invariants. Let $\G_d$ 
denote the (normal) cyclic subgroup of $W_d$ generated by $ts=\diag(\z,\z^{-1})$. 
It is easily seen that 
$$\CM[V \times V^*]^{\G_d}=\CM[xy,XY,xX,yY,(x^{d-j}Y^j)_{0 \le j \le d},(y^{d-j}X^j)_{0 \le j \le d}].$$
Using the action of $W_d/\G_d \simeq \ZM/2\ZM$ on $\CM[V \times V^*]^{\G_d}$, this can be rewritten
\equat\label{eq:gammad}
\CM[V \times V^*]^{\G_d}=\CM[V \times V^*]^{W_d}[\d,\b_0,\b_1,\dots,\b_d].
\endequat
As the product of two elements of the list $\d$, $\b_0$, $\b_1$,\dots, $\b_d$ 
is a $W_d$-invariant, we get that 
\equat\label{eq:module}
\text{The $\CM[V \times V^*]^{W_d}$-module $\CM[V \times V^*]^{\G_d}$ 
is generated by $1$, $\d$, $\b_0$, $\b_1$,\dots, $\b_d$.}
\endequat
Since $\CM[V\times V^*]^{\G_d}=\CM[V \times V^*]^{W_d} \oplus \CM[V \times V^*]^{W_d-\text{sem}}$, 
we get that
\equat\label{eq:module-bis}
\text{The $\CM[V \times V^*]^{W_d}$-module $\CM[V \times V^*]^{W_d-\text{sem}}$ 
is generated by $\d$, $\b_0$, $\b_1$,\dots, $\b_d$.}
\endequat
Therefore, the ideal $I=(\d \CM[V \times V^*])^{W_d}=\d\CM[V \times V^*]^{W_d-\text{sem}}$ 
of $\CM[V \times V^*]^{W_d}$ is generated by $\d^2$, $\d\b_0$, $\d\b_1$,\dots, $\d\b_d$. 

Write $D= e^2-4qQ = \d^2$ and consider the following $W_d$-invariant rational functions:
$$
b_j=\cfrac{\b_j}{\d}=\cfrac{x^{d-j}Y^j-y^{d-j}X^j}{xX-yY}, \quad  0 \le j \le d.
$$
Then $Db_j=\d\b_j$ is $W_d$-invariant and $I$ is the ideal of 
$\CM[V \times V^*]^{W_d}$ generated by $D$, $Db_0$, $Db_1$,\dots, $Db_d$. 
Let $\QCt(d)$ be the blowup of $\QC(d)$ along this ideal. 
The (contracting) $\CM^\times$-action on $V \times V^*$ by homothety induces a 
contracting action on the quotient $\QC(d)$. For this action, $D$ and the $b_j$'s are homogeneous 
(of degree $2$ and $d-2$ respectively) so the blowup $\QCt(d)$ inherits a $\CM^\times$-action. 

Recall that the blowup of an affine variety $\Spec(A)$ along an ideal $\aG$ generated by elements 
$f_1$,\dots, $f_r$ admits an affine open covering indexed by the $f_j$'s and such that the affine open subset associated with $f_j$ is isomorphic to $\Spec(A[\aG/f_j])$. In our case, the variety $\QCt(d)$ is covered by affine open subsets $\YC(d)$, $\YC_0$, $\YC_1$,\dots, $\YC_d$ 
defined by $\YC(d)=\Spec(\CM[V \times V^*]^{W_d}[I/D])$ and $\YC_j=\Spec(\CM[V \times V^*]^{W_d}[I/(Db_j)])$, for $j=0, 1, \cdots, d$.

They are all $\CM^\times$-stable.
The following relations follow from straightforward computations:
\equat\label{eqbi}
\begin{cases}
a_j = e b_j -2qb_{j+1} &\text{if $0 \le j \le d-1$},\\
a_j= 2 Q b_{j-1} - eb_j &\text{if $1 \le j \le d$,}\\
\end{cases}
\endequat
\equat\label{eqDbi}
\begin{cases}
Db_0=ea_0-2qa_1,\\
Db_j = Qa_{j-1} - q a_{j+1} &\text{if $1 \le j \le d-1$}, \\
Db_d = 2Qa_{d-1} - ea_d.
\end{cases}
\endequat
Let $\eta$ be a primitive $4d$-th root of unity. The action of $\diag(\eta,\eta^{-1})$ sends $a_j$ to $\sqrt{-1} \beta_j$, while $e$, $q$ and $Q$ remain invariant. Applying this transformation to the equations \eqref{eq:Q(d)} describing $\QC(d)$, and dividing the first line by $\delta$ and the second line by $\delta^2$, we obtain the following relations (note the sign change in the second line!):
\equat\label{eq:bi}
\begin{cases}
e b_j = q b_{j+1} + Q b_{j-1}\\
b_j b_k - b_{j-1} b_{k+1}  = q^{d-k-1} Q^{j-1} \Psi_{k-j}(q,Q,e)
\end{cases}\quad \text{if $1\le j \le k \le d-1$.}
\endequat

\begin{lem}\label{lem:blowup}
With the above notation, we have:
\begin{itemize}
\itemth{a} $\YC_0 \simeq \YC_d \simeq \CM^4$.

\itemth{b} If $1 \le r \le d-1$, then $\YC_r$ is smooth.
%
\end{itemize}
\end{lem}

\begin{proof}
(a) We claim that $\YC_0 ={\rm Spec}(\mathbb{C}[q, a_0, \frac{1}{b_0}, \frac{b_1}{b_0}])$, which will 
imply that $\YC_0 \simeq \CM^4$ as $\YC_0$ is of dimension 4. For this, we will express other variables 
in terms of $q$, $a_0$, $\frac{1}{b_0}$, $\frac{b_1}{b_0}$.

By~\eqref{eqbi}, $a_0 = eb_0 - 2qb_1$, whence $e=\frac{1}{b_0} a_0 + 2 q \frac{b_1}{b_0}$, 
which gives the expression for $e$.

Take $j=k=1$ in~\eqref{eq:bi}, we have $b_1^2 - b_0 b_2 = q^{d-2}$ as $\Psi_0=1$, hence  
$\frac{b_2}{b_0} = (\frac{b_1}{b_0})^2 - q^{d-2} (\frac{1}{b_0})^2$, which gives the expression 
for $\frac{b_2}{b_0}$.

By~\eqref{eq:bi}, we have $eb_1=qb_2+Qb_0$, which gives $Q = e \frac{b_1}{b_0} - q \frac{b_2}{b_0}$,  
the expression for $Q$, as $\frac{b_2}{b_0} $ is a polynomial in $q, a_0, \frac{1}{b_0}, \frac{b_1}{b_0}$.

By~\eqref{eq:bi}, we have $b_1 b_k - b_0 b_{k+1} = q^{d-k-1} \Psi_{k-1}(q,Q,e)$, 
which gives inductively the expressions for $\frac{b_k}{b_0}$ for all $k \geq 3$.  

For $1 \leq j \leq d-1$, we have $a_j = eb_j-2q b_{j+1}  $.  Replacing 
$e=\frac{1}{b_0} a_0 + 2 q \frac{b_1}{b_0}$ and using \eqref{eq:bi}, we have
$$a_j = a_0 \frac{b_j}{b_0} + 2q \frac{b_1b_j-b_0b_{j+1}}{b_0} = a_0 \frac{b_j}{b_0} + 2q^{d-j} 
\Psi_{j-1}(q,Q,e)  \frac{1}{b_0}. $$
It remains to express $a_d$. Recall that $a_d = 2Q b_{d-1} - e b_d$ and $a_1 = 2 Qb_0 - eb_1$ 
by~\eqref{eqbi}. This gives
$$
a_d = 2Qb_0 \frac{b_{d-1}}{b_0} - eb_d = a_1 \frac{b_{d-1}}{b_0} 
+ e \frac{b_1b_{d-1}-b_0b_d}{b_0} =a_1 \frac{b_{d-1}}{b_0} + e \Psi_{d-2}(q,Q,e) \frac{1}{b_0},
$$
as desired. This completes the proof for $\YC_0$ (and the proof for $\YC_d$ is 
similar).

\medskip

(b) For ease of notation, let $B_r=\frac{1}{b_r}$, $B_j=\frac{b_j}{b_r}$ for $0\le j\le d$, $j\neq r$.
By~\eqref{eqbi} we have $a_r=eb_r-2qb_{r+1}$, hence $e=a_rB_r +2qB_{r+1}$.
(We also have $e=2QB_{r-1}-a_rB_r$.)
For $j>r+1$, the equation $b_{r+1}b_{j-1}-b_rb_j=q^{d-j}Q^r\Psi_{j-r-2}(q,Q,e)$ allows us to 
iteratively express $B_j$ in terms of $q,Q$, and $e$; 
we can deal with $B_0,\ldots ,B_{r-2}$ similarly.

For $0\le j < r$ we substitute $e=a_r\cdot\frac{1}{b_r}+2q\frac{b_{r+1}}{b_r}$ into~\eqref{eqbi} to obtain:
$$a_j=\left(a_r\frac{1}{b_r}+2q\frac{b_{r+1}}{b_r}\right)b_j-2qb_{j+1} 
= a_rB_j+2q^{d-j+2}Q^r\Psi_{j-r-3}(q,Q,e)B_r.$$
Similarly, we substitute $e=2QB_{r-1}-a_rB_r$ into~\eqref{eqbi} to express 
$a_{r+1},\ldots ,a_d$ in terms of $a_r$, $q$, $Q$ and $e$.
Thus we have obtained 
${\mathbb C}[q,Q,e,a_0,\cdots ,a_d,B_0,\cdots,B_d]={\mathbb C}[q,Q,a_r,B_r,B_{r-1},B_{r+1}]$.
We have two relations: 
\equat\label{Yrrels}
qB_{r+1}-QB_{r-1}+a_rB_r=0\quad\text{and}\quad 
\quad B_{r-1}B_{r+1}+q^{d-r-1}Q^{r-1}B_r^2-1=0.
\endequat
The first of these follows from the equality of the two expressions for $e$; 
the second from~\eqref{eq:bi}.
It is straightforward to check that the Jacobian has rank 2 (under the same conditions on 
$q,Q,a_r,B_r,B_{r-1},B_r$), which implies that $\YC_r$ is smooth (of dimension 4). More specifically, it implies that the subvariety of ${\mathbb A}^6$ with 
equations~\eqref{Yrrels} is smooth of dimension 4; by dimensions we know that $\YC_r$ is 
isomorphic to one of its irreducible components.
\end{proof}


\subsection{Singularity of the blowup}
By Lemma~\ref{lem:blowup}, the singularities of the blowup $\QCt(d)$ 
are contained in the affine open subset $\YC(d)$. First, 
by definition,  $\CM[\YC(d)] = \CM[V \times V^*]^{W_d}[b_0,b_1,\dots,b_d]$. 
Using~\eqref{eqbi}, we see that
\equat\label{eq:yd}
\CM[\YC(d)] = \CM[q,Q,e,b_0,b_1,\dots,b_d].
\endequat
This shows that $\YC(d)$ is a closed subvariety of $\CM^{d+4}$ 
and the $\CM^\times$-action on $\YC(d)$ is contracting (to $0$). 
Here is a presentation of $\YC(d)$ as well as a first result 
on its singularities:


\begin{prop}\label{prop:blowup}
Equations for $\YC(d) \subset \CM^{d+4}$ are given by~\eqref{eq:bi}. 
Moreover, $\YC(d)$ is normal and $0$ is its only singular point. 
\end{prop}


\begin{proof}
Let $\YC$ be the affine scheme defined by the equations~\eqref{eq:bi}. 
These equations being homogeneous (recall that $e$, $q$ and $Q$ have degree $2$ 
and $b_0$, $b_1$,\dots, $b_d$ have degree $d-2$), $\CM[\YC]$ is $\NM$-graded, 
with $0$-component equal to $\CM$. Equations~\eqref{eq:bi} show 
that we have a surjective morphism $\CM[\YC] \longto \CM[\YC(d)]$, hence 
a closed embedding $\YC(d) \injto \YC$. 

Let $\widehat{\CM[\YC]}_0$ denote the completion of $\CM[\YC]$ 
at the point $0$ and, similarly, let $\widehat{\CM[\ZC(d)]}_0$ 
denote the completion of $\CM[\ZC(d)]$ at the point $0$. 
The comparison of the definition of $\ZC(d)$ with the equations~\eqref{eq:bi} 
shows that we have an isomorphism 
$$\widehat{\CM[\YC]}_0 \simeq \widehat{\CM[\ZC(d)]}_0\leqno{(\#)}$$
given by $e \mapsto e$, $q \mapsto q$, $Q \mapsto Q$ and 
$b_j \mapsto a_j/\sqrt{d^2 + 4qQ - e^2}$ for $0 \le j \le d$. 

By Theorem~\ref{theo:zd}, the variety $\ZC(d)$ is normal so $\widehat{\CM[\ZC(d)]}_0$ 
is a normal domain by Zariski's main Theorem. This shows that $\widehat{\CM[\YC]}_0$ 
is a normal domain. As the natural map $\CM[\YC] \longto \widehat{\CM[\YC]}_0$ is injective 
(because $\CM[\YC]$ is $\NM$-graded and $\widehat{\CM[\YC]}_0$ is the completion 
with respect to the unique maximal homogeneous ideal), this shows that 
$\CM[\YC]$ is a domain. Therefore, 
$${\mathrm{Kdim}}~\widehat{\CM[\YC]}_0 ={\mathrm{Kdim}}~\widehat{\CM[\ZC(d)]}_0
={\mathrm{Kdim}}~\CM[\ZC(d)]=4,$$
where ${\mathrm{Kdim}}$ denotes the Krull dimension. Hence $\YC$ is irreducible 
reduced of dimension $4$. 
As $\YC(d)$ is irreducible of dimension $4$, this shows that $\YC=\YC(d)$. In other words,  
equations for $\YC(d) \subset \CM^{d+4}$ are given by~\eqref{eq:bi}. Moreover, 
$(\#)$ becomes 
\equat\label{eq:equivalence}
\widehat{\CM[\YC(d)]}_0 \simeq \widehat{\CM[\ZC(d)]}_0
\endequat

In particular, the singularities $(\ZC(d),0)$ and $(\YC(d),0)$ 
are locally analytically isomorphic, so $0$ is an isolated singularity of $\YC(d)$. 
Since the $\CM^\times$-action on $\YC(d)$ is contracting, any irreducible 
component of the singular locus of $\YC(d)$ must contain $0$. 
Thus, $0$ is the only singularity of $\YC(d)$.

Finally, $\YC(d)$ is normal at $0$ because its completion at $0$ is normal, 
so $\YC(d)$ is normal because all other points are smooth. The proof 
of the proposition is complete.
\end{proof}


When viewed inside $\QCt(d)$, the point $0$ of $\YC(d)$ will still be denoted by $0$, as in the Introduction. 


\begin{coro}\label{coro:blowup}
The blowup $\QCt(d)$ is normal and admits a unique singular 
point, namely the point $0 \in \YC(d)$. The singularities 
$(\QCt(d),0)$ and $(\ZC(d),0)$ are locally analytically isomorphic.
\end{coro}


\begin{rema}\label{rem:hilbert}
(1) If $V$ is replaced by $\CM^n$ and $W_d$ is replaced by the symmetric group $\SG_n$, 
then the analogue of $\QCt(d)$ is the Hilbert scheme of $n$ points in the plane by 
results of Haiman~\cite[Proposition~2.6]{haiman}. This gives the symplectic resolution of 
$(\CM^n \times \CM^n)/\SG_n$.

\medskip

(2) The fact that $\QCt(5)$ has an isolated singularity 
has been observed in \cite[\S{12.3.1}]{FJLS}.\finl
\end{rema}


\section{Other descriptions of the blowup}


\subsection{Orbit closures}
We denote by $\Gb\Sb\pb(V \times V^*)$ the general symplectic group of $V \times V^*$, which is the group of linear automorphisms of $V \times V^*$ which preserve 
the symplectic form up to a scalar. 
Since $W_d$ is a Coxeter group, the finite group $W_d$ commutes with a subgroup of 
$\Gb\Sb\pb(V\times V^*)$ isomorphic to $\Gb\Lb_2=\Gb\Lb_2(\CM)$ (see for instance~\cite[\S{3.6}]{calogero}): 
to see this, note that $V \simeq V^*$ as a $W_d$-module, so $V \times V^* \simeq \CM^2 \otimes V$ 
with $W_d$ acting trivially on $\CM^2$ and the action of $\Gb\Lb_2$ is on the left hand side of this 
tensor product decomposition. Note that the restriction of this action to $\CM^\times$ (identified 
with the center of $\Gb\Lb_2$) is by homothety, so is the action considered in~\S\ref{sub:blowup}.


\begin{rema}\label{rem:explicit}
Note that the above action of $\Gb\Lb_2$ on $V \times V^*$ does not coincide with 
the action of $\Gb\Lb_\CM(V)=\Gb\Lb_2(\CM)$ given by the natural action on $V$ and 
the contragredient action on $V^*$. To distinguish both actions, the natural 
module for $\Gb\Lb_2$ will be denoted by $U_2$ (as in the subsection about conventions/notation 
in the Introduction) and we use the notation $\Gb\Lb_2$ (instead of $\Gb\Lb_2(\CM)$) 
when we talk about the action respecting the symplectic form on $V \times V^*$ up to scalar. 
The canonical basis of $U_2$ will be denoted by $(\e_1,\e_2)$. 

Here, an explicit isomorphism of $W_d$-modules $V \longiso V^*$ 
is given for instance by $x \mapsto Y$, $y \mapsto X$. This leads to the following formulas:
$$
\text{Action of
$\begin{pmatrix} a & b \\ c & d \end{pmatrix} \in \Gb\Lb_2$ on $V \times V^*$ $~\rightsquigarrow~$} 
\begin{cases}
x \longmapsto ax + cY, \\
y \longmapsto ay + c X, \\
X \longmapsto by+dX,\\
Y \longmapsto bx + dY. \\
\end{cases}
$$
Note that the action of $\Sb\Lb_2$ is symplectic while 
the action of ${\mathbb C}^\times$ scales the symplectic form 
(and so scales the Poisson bracket).\finl
\end{rema}


Hence $\Gb\Lb_2$ acts on $\CM[q,Q,e,a_0,\ldots ,a_d]$; by a straightforward calculation, 
$q$, $Q$, $e$ span a copy of the adjoint representation $\sG\lG_2$ tensored with the determinant, 
and $a_0$, $a_1$,\dots, $a_d$ span a copy of $\Sym^d(U_2)$ where $U_2$ is the natural representation 
for $\Gb\Lb_2$. We denote by $\sG\lG_2^{(1)}$ the above representation of $\Gb\Lb_2$: its underlying 
space is $\sG\lG_2$ (and we will forget the exponent $^{(1)}$ whenever we forget 
the $\Gb\Lb_2$-action) but the usual adjoint action of $\Gb\Lb_2$ on $\sG\lG_2$ 
is multiplied with the determinant. This discussion shows that $\QC(d)$ 
may be viewed as a $\Gb\Lb_2$-stable closed subvariety of $\sG\lG_2^{(1)} \oplus \Sym^d(U_2)$. 

This action is very useful for understanding the geometry of $\QC(d)$.
While there are infinitely many orbits for the action of $\Sb\Lb_2$ on its adjoint representation, 
there are only two non-zero orbits for the above action of $\Gb\Lb_2$ on $\mathfrak{sl}_2^{(1)}$, 
with representatives 
$$H=\begin{pmatrix} 1 & 0 \\ 0 & -1 \end{pmatrix}\quad\text{and}\quad 
E=\begin{pmatrix} 0 & 1 \\ 0 & 0 \end{pmatrix}.$$
In particular, the $\Gb\Lb_2$-orbit of any non-zero semisimple element is dense in 
$\mathfrak{sl}_2^{(1)}$. By the same token:


\begin{lem}\label{orblem1}
Viewing $\QC(d)$ as a $\Gb\Lb_2$-stable closed subvariety of $\sG\lG_2^{(1)} \oplus \Sym^d(U_2)$ 
as above, we have:
\begin{itemize}
\itemth{a} The variety $\QC(d)$ is the closure of the $\Gb\Lb_2$-orbit 
of $(H,\e_1^d+\e_2^d)$ in $\mathfrak{sl}_2^{(1)} \oplus \Sym^d(U_2)$.

\itemth{b} If $d$ is odd (resp. even), then the singular locus of $\QC(d)$ is the closure of the 
$\Gb\Lb_2$-orbit of $(E,\e_1^d)$ (resp. of the union of the $\Gb\Lb_2$-orbits of 
$(E,\e_1^d)$ and $(E,-\e_1^d)$). 
\end{itemize}
\end{lem}

\begin{proof}
It follows immediately from the generating set for the ring of invariants and the above 
remarks that $(V\times V^*)/W_d$ is isomorphic (via a closed immersion) to a Zariski 
closed $\Gb\Lb_2$-stable subset of $\mathfrak{sl}_2^{(1)} \oplus \Sym^d(U_2)$.
Inspecting the relations for the generators, we see that if $q=Q=0$ and $e=1$ 
(corresponding to a non-zero multiple of $H$ in $\mathfrak{sl}_2$) then $a_j=0$ for $1\le j\le d-1$ 
and $a_0a_d=1$.
It is easy to see that the stabilizer of the point $(H,\e_1^d+\e_2^d)=(0,0,1,1,0,\ldots ,0,1)$ in 
$\Gb\Lb_2$ is finite, which (after suitably scaling a basis 
of $\Sym^d(U_2)$ if necessary) proves (a).


To establish (b), note that the singular locus has dimension 2, and all non-zero $\Gb\Lb_2$-orbits 
have dimension $2$ or $4$; the orbit is dense 
if and only if the projection to $\sG\lG_2^{(1)}$ is semisimple.
Thus, the singular locus is a union of orbits of elements of the form $(E,u)$ or $(0,u)$.
Now we recall that the singular locus is the image in $(V\times V^*)/W_d$ of the points with 
$x=\zeta^k y$, $Y=\zeta^k X$, in which case $q=\zeta^k y^2$, $e=2\zeta^k yX$, $Q=\zeta^k X^2$ 
and $a_j=2y^{d-j}X^j$.

If $q=e=Q=0$ then clearly $a_j=0$ for all $j$, so there are no non-zero orbits in the singular 
locus of the form $(0,u)$; for the orbits of the form $(E,u)$, we assume $q=e=0$ so that $y=0$.
Then $Q=\zeta^k X^2$, $a_0=\ldots = a_{d-1}=0$ and $a_d=2X^d$.
In particular, $a_d^2=4Q^d$, so there are at most two such orbits for each fixed non-zero value of $Q$; 
if $d$ is odd then the two different possibilities for $a_d$ are obviously conjugate via the 
$\mub_2$-action; if $d$ is even then $2Q^{d/2}=\pm a_d$ accordingly as 
$k$ is even or odd, so we do obtain two disjoint orbits.
\end{proof}


An immediate computation shows that $g(\d)=\det(g)\d$ for $g \in \Gb\Lb_2$ 
(this will be better explained in Remark~\ref{rem:d=4}) 
and that the subspace $\CM \b_0 \oplus \CM \b_1 \oplus \cdots \CM \b_d$ is $\Gb\Lb_2$-stable 
and is isomorphic to the representation $\Sym^d(U_2)$. In particular, $\d$ 
is $\Sb\Lb_2$-invariant (note that $D=\d^2$ corresponds to the determinant in $\sG\lG_2$). 
It follows that $\Gb\Lb_2$ acts on the blowup $\QCt(d)$ and stabilizes the 
affine open subset $\YC(d)=\Spec \CM[q,Q,e,b_0,\ldots ,b_d]$. Now, 
the previous remark shows that $\CM b_0 \oplus \CM b_1 \oplus \cdots \CM b_d$ 
is $\Gb\Lb_2$ stable and is isomorphic to the representation $\Sym^d(U_2)^{(-1)}$, 
where the exponent $^{(-1)}$ means that the action on $\Sym^d(U_2)^{(-1)}$ 
is the natural action on $\Sym^d(U_2)$ tensored with the inverse of the determinant. 
In particular, for the $\CM^\times$-action, $b_0$, $b_1$,\dots, $b_d$ have weight $d-2$. 

This discussion shows that $\YC(d)$ may be viewed as a $\Gb\Lb_2$-stable closed 
subvariety of $\sG\lG_2^{(1)} \oplus \Sym^d(U_2)^{(-1)}$. 


\begin{lem}\label{orblem2}
Through this embedding, $\YC(d)$ is the closure of the $\Gb\Lb_2$-orbit of $(H,\e_1^d-\e_2^d)$.
\end{lem}

\begin{proof}
Inspecting the relations \eqref{eq:bi}, we see that if $e=1$ and $q=Q=0$ then $b_j=0$ for 
$1\le j\le d-1$ and $b_0b_d=-1$.
We can now complete the proof with the same argument as in Lemma \ref{orblem1}(a).
\end{proof}


\begin{rema}\label{rem:d=4}
Via the inclusion $\sG\lG_2 \simeq \sG\oG_3 \subset \sG\lG_3$, we have a $\Gb\Lb_2$-equivariant isomorphism
\[
\begin{array}{ccc}
\sG\lG_2^{(1)} \oplus \Sym^4(U_2)^{(-1)} & \longto & \sG\lG_3\\[1em]
E & \longmapsto & \begin{psmallmatrix} 0 & \sqrt 2 & 0 \\ 0 & 0 & -\sqrt 2 \\ 0 & 0 & 0 \end{psmallmatrix}\\[1em]
H & \longmapsto & \begin{psmallmatrix} 2 & 0 & 0 \\ 0 & 0 & 0 \\ 0 & 0 & -2 \end{psmallmatrix}\\[1em]
F & \longmapsto & \begin{psmallmatrix} 0 & 0 & 0 \\ \sqrt 2 & 0 & 0 \\ 0 & -\sqrt 2 & 0 \end{psmallmatrix}\\[1em]
\e_1^4 & \longmapsto & \begin{psmallmatrix} 0 & 0 & 2 \\ 0 & 0 & 0 \\ 0 & 0 & 0 \end{psmallmatrix}\\[1em]
\e_1^3 \e_2 & \longmapsto & \begin{psmallmatrix} 0 & \sqrt 2 / 2 & 0 \\ 0 & 0 & \sqrt 2 / 2 \\ 0 & 0 & 0 \end{psmallmatrix}\\[1em]
\e_1^2 \e_2^2 & \longmapsto & \begin{psmallmatrix} -1/3 & 0 & 0 \\ 0 & 2/3 & 0 \\ 0 & 0 & -1/3 \end{psmallmatrix}\\[1em]
\e_1 \e_2^3 & \longmapsto & \begin{psmallmatrix} 0 & 0 & 0 \\ -\sqrt 2 / 2 & 0 & 0 \\ 0 & -\sqrt 2 / 2 & 0 \end{psmallmatrix}\\[1em]
\e_2^4 & \longmapsto & \begin{psmallmatrix} 0 & 0 & 0 \\ 0 & 0 & 0 \\ 2 & 0 & 0 \end{psmallmatrix}\\[1em]
\end{array}
\]
In particular,
\(
H + \e_1^4 - \e_2^4  \longmapsto  \begin{psmallmatrix} 2 & 0 & 2 \\ 0 & 0 & 0 \\ -2 & 0 & -2 \end{psmallmatrix}
\), which is of rank $1$ and trace $0$, i.e.~an element of the minimal nilpotent orbit.

Now, the $\Gb\Lb_2$-orbit of $H + \e_1^4 - \e_2^4$ has dimension $4$ by Lemma~\ref{orblem2} 
and is contained in the minimal nilpotent orbit which has also dimension $4$. 
This proves that $\YC(4)$ is isomorphic to the closure of the minimal nilpotent 
orbit of $\sG\lG_3$.
This provides a {\sc Magma} free proof of Theorem~\ref{theo:zd}(b) and~(d).\finl
\end{rema}


\subsection{Singular locus of ${\boldsymbol{\QC(d)}}$}

In the previous subsection, we have  identified $\QC(d)$ with a closed subvariety of 
$\sG\lG_2^{(1)} \oplus \Sym^d(U_2)$. It is easily seen that $I=(\d\CM[V \times V^*])^{W_d}$ 
is the ideal vanishing on elements of $\QC(d)$ whose projection onto $\sG\lG_2$ is nilpotent. 
Now, let $I_\sing$ denote the ideal of $\CM[V \times V^*]^{W_d}$ vanishing on the reduced 
singular locus $\QC_\sing(d)$ of $\QC(d)$.

%

\begin{lem}\label{singlem}
The ideal $I_\sing$ is generated by $I$ and $(\b_j\b_k)_{0 \le j,k \le d}$. 

\end{lem}

\begin{proof}
Note that $\b_j\b_k=Db_jb_k$ obviously vanishes at any singular point.
We will use the action of $\Sb\Lb_2$ on the coordinate ring; 
clearly $I_\sing$ and the singular locus are stable under this action.

Since $\b_j\b_k \equiv a_ja_k$ modulo ${\mathbb C}[q,Q,e]$, it follows that 
${\mathbb C}[V\times V^*]^{W_d}/I_\sing$ is generated as a 
${\mathbb C}[q,Q,e]/(e^2-4qQ)$-module by $1, a_0,\ldots ,a_d$.
Now ${\mathbb C}[q,Q,e]/(e^2-4qQ)$ is the coordinate ring of the nullcone of $\mathfrak{sl}_2$; 
by a well-known theorem of Kostant~\cite[Theorem~0.9]{kostant}, 
it contains one copy of (the $\Sb\Lb_2$-module) $\Sym^{2l}(U_2)$ 
in each even degree $2l$.
Inspecting the relations~\eqref{eqbi} and using the fact that $e^2-4qQ\in I_\sing$, we find that 
the following monomials form a basis of the submodule of ${\mathbb C}[V\times V^*]^{W_d}/I_\sing$ 
spanned by $a_0,\ldots ,a_d$:
$$a_0, \ldots , a_d; \; qa_0, ea_0, ea_1, \ldots , ea_d, Qa_d; \; q^2 a_0, qea_0, qQa_0, qQa_1, \ldots, qQa_d, Qea_d, Q^2 a_d; \ldots$$
and in particular there is one copy of $\Sym^{d+2l}U$ in each degree $d+2l$.
For $d$ odd, this is especially useful: the quotient ${\mathbb C}[V\times V^*]^{W_d}/I_\sing$ 
decomposes over $\Sb\Lb_2$ as a sum of non-isomorphic irreducible submodules, 
each one appearing in its own degree; for $d$ even, the submodules $\Sym^{2l}(U_2)$ with $2l\leq d$ are doubled.
By the $\Sb\Lb_2$-action, it only remains to determine which highest weight vectors vanish on the 
singular locus.
These highest weight vectors are: $Q^j$ and $Q^j a_d$ for $d$ odd; $Q^j$ and any linear span of 
$Q^j a_d, Q^{j+d/2}$ if $d$ is even.
Our result will be proved if we can establish that none of these highest weight vectors vanishes 
on $\QC_\sing(d)$.

Recall that the singular locus is the set of orbits of points satisfying 
$x=\zeta^k y, Y=\zeta^k X$ for some $k$. 
Then $Q=\zeta^k X^2$ and $a_d=2X^d$.
None of these vanish, so for $d$ odd, we obtain our result immediately.
If $d=2m$ is even then $Q^m=\pm X^d$ accordingly as $k$ is even or odd; 
these two possibilities correspond to the two irreducible components of $\QC_\sing(d)$.
It follows that there do exist highest weight vectors in degree $d+2l$ vanishing 
on each of these irreducible components (specifically, $Q^l(a_d\mp 2Q^m)$).
However, no non-zero linear span of $Q^l a_d$ and $Q^{m+l}$ can vanish on the {\it whole} 
singular locus, so our proof is complete.
\end{proof}

\begin{coro} \label{coro:singlocus}
For $d=2m$ even, the ideals of the two irreducible components of the singular locus are:
$$J_1 = (D,a_0-2q^m, a_1-q^{m-1}e, a_2-2q^{m-1}Q, \ldots , a_d-2Q^m), \quad \mbox{and}$$
$$J_2 = (D,a_0+2q^m, a_1+q^{m-1}e, a_2+2q^{m-1}Q, \ldots , a_d+2Q^m).$$
\end{coro}

For ease of notation, let $a_{2j}^{\pm}: = a_{2j}\pm 2q^{m-j}Q^j$ and $a_{2j+1}^{\pm}:= a_{2j+1}\pm q^{m-j-1}Q^{j+1}e$.
Note that: 
$$a_{2j}^{\pm} = (x^{m-j}Y^j\pm y^{m-j}X^j)^2\;\; \mbox{and}$$
$$a_{2j+1}^{\pm} = (x^{m-j}Y^j\pm y^{m-j}X^j)(x^{m-j-1}Y^{j+1}\pm y^{m-j-1}X^{j+1}),$$
so that the generators of $J_1$ (resp. $J_2$) can be expressed naturally in terms of 
the generating invariants (resp. semi-invariants) for the dihedral group of order $2m$.

\begin{proof}
By the proof of Lemma \ref{singlem}, it will suffice to show that $I_\sing \subset J_1$ and $I_\sing \subset J_2$.
It follows easily from \eqref{eqbi} that $Db_j\in J_1$ and $Db_j\in J_2$.
Since $Db_jb_k-Db_{j-1}b_{k+1}\in (D)$, it will therefore suffice to show that each of $Db_0^2, Db_0b_1, Db_1^2, \cdots , Db_d^2$ belongs to $J_1\cap J_2$.
We recall that $$Db_{2j}^2 = (x^{2m-2j}Y^{2j}-y^{2m-2j}X^{2j})^2 =  a_{2j}^{-}a_{2j}^+\in J_1\cap J_2.$$
Similarly, $$ Db_{2j+1}^2 = \frac{1}{4}\left(a_{2j}^+a_{2j+2}^-+2a_{2j+1}^+a_{2j+1}^-+a_{2j}^-a_{2j+2}^+\right)\;\;\mbox{and}\;\; Db_jb_{j+1}=\frac{1}{2}\left( a_j^+a_{j+1}^-+a_j^-a_{j+1}^+\right).$$
\end{proof}

\begin{prop}
The blowup $\QCt(d)$  of $\QC(d)$ at $I$ is isomorphic to the blowup of $\QC(d)$ at its reduced 
singular locus. Moreover, the symplectic form on $V \times V^*$ induces a symplectic 
structure on $\QCt(d) \setminus \{0\}$.
\end{prop}

\begin{proof}
The first assertion follows from the description of the blowup at $I_\sing$ as the closure of the 
set of elements of ${\mathbb A}^{d+4}\times {\mathbb P}^{(d+1)(d+2)/2}$ of the form 
$$((q,Q,e,a_0,\ldots ,a_d),[D:Db_0:\ldots : Db_d: Db_0^2: Db_0b_1: \ldots : Db_d^2]),$$
where at least one of $D, Db_0, \ldots , Db_d^2$ is non-zero; the projection to 
${\mathbb P}^{(d+1)(d+2)/2}$ is just the Veronese embedding applied to 
$[D:Db_0:\ldots : Db_d]\in{\mathbb P}^{d+1}$.


Let us now prove the second statement. We can consider $\QC(d)$ as the $\mub_2$-quotient of 
$(V\times V^*)/\Gamma_d$. Since there are no non-zero fixed points for the action of 
$\Gamma_d$, the quotient  $(V\times V^*)/\Gamma_d$ has an isolated singularity.
For $p \in \QC_\sing(d) \setminus\{0\}$,  $\QC(d)$ has $A_1$-singularity at $s$.  
As $\pi: \QCt(d) \to \QC(d)$ is the blowup of the singular locus,  it restricts to the minimal 
resolution to $\QC(d) \setminus\{0\}$.  This implies that $\QCt(d) \setminus \pi^{-1}(0)$ 
admits a symplectic structure (say $\tilde{\omega}$) coming from  that on $V \times V^*$.  
As $\pi^{-1}(0)$  has dimension $2$,  $\tilde{\omega}$ extends to a symplectic structure on the 
smooth locus of $\QCt(d)$, which is just $\QCt(d) \setminus \{0\}$ by Corollary~\ref{coro:blowup}.
\end{proof}


By Corollary~\ref{coro:blowup}, the singularities $(\QCt(d),0)$ and $(\ZC(d),0)$ 
are analytically isomorphic but this does not ensure that the variety $\QCt(d)$ 
inherits a symplectic form on its smooth locus. Thanks to the previous 
proposition, we can now deduce:

%


\begin{coro}\label{coro:symplectic}
The variety $\QCt(d)$ has  an isolated  symplectic singularity.
\end{coro}


 The singularity $\QCt(d)$ does not admit a projective symplectic resolution. 
 Since the blowup morphism is crepant (being Poisson) it would imply by composition that $(V \times V^*)/W_d$ 
 admits a projective symplectic resolution, which is not the case.  In fact, 
 this also follows from the general fact that the only $4$-dimensional isolated symplectic 
 singularity admitting a symplectic resolution is analytically isomorphic to the minimal 
 nilpotent orbit closure in $\sG\lG_3$ (see~\cite[Theorem~1.1]{WW}).
 
 Note that the symplectic singularity $(\ZC(d),0)$ does not admit a contracting $\CM^\times$-action. 
 Therefore, the fact that symplectic singularity $(\YC(d),0)$ is locally analytically isomorphic 
 to $(\ZC(d),0)$ and admits a contracting $\CM^\times$-action can be viewed as 
 confirmation of a general conjecture of Kaledin~\cite[Conjecture~1.8]{KaledinSurvey} 
 in this particular case.


\section{Local fundamental group}

In order to compute the local fundamental group of $\YC(d)$ around $0$, 
we will consider the following smooth irreducible 
surface 
$$\SC=\{(x,y,z) \in \CM^3~|~x^2- y^2z=1\}.$$ 
We define the morphism $\ph : \SC \longto \YC(d) \setminus \{0\}$ by
$$\ph(x,y,z)=(1,-z,0,y,x,yz,xz,\dots,\underbrace{yz^k}_{b_{2k}},
\underbrace{xz^k}_{b_{2k+1}},\dots).$$
It's easy to see that $\ph$ is a closed immersion, whose image is equal to 
$$\{(q,Q,e,b_0,b_1,\dots,b_d) \in \YC(d)~|~q=1\text{~and~}e=0\}.$$


\begin{lem}\label{lem:pi1-x}
The smooth surface $\SC$ is simply-connected.
\end{lem}


\begin{proof}
Let $p_0=(1,1,0) \in \SC$. 
Let $\SC^\circ$ denote the open subset $\{(x,y,z) \in \SC~|~y \neq 0\}$. 
Then $p_0 \in \SC^\circ$ and the map $\pi_1(\SC^\circ,p_0) \longto \pi_1(\SC,p_0)$ 
is surjective~\cite[Theorem~2.3]{godbillon}. Now, $\CM \times \CM^\times \simeq \SC^\circ$ 
through the variables $(x,y)$ and the map $(x,y) \longmapsto (x,y,y^{-2}(x^2-1))$, 
so $\pi_1(\SC^\circ,p_0)$ is generated by the loop $\g$ defined by 
$\g(t)=(1,e^{2\sqrt{-1}\pi t},0)$. 
Then it remains to show that $\g$ is homotopic, in $\SC$, to the trivial loop. 
But $\g$ is contained in $\{1\} \times \CM \times \{0\} \subset \SC$, 
so the result follows.
\end{proof}

Coming back to our aim of proving Theorem~\ref{theo:main 1}, 
it remains to prove that the local fundamental group of the singularity 
$(\YC(d),0)$ is trivial. Moreover, the existence of a contracting $\CM^\times$-action 
implies that the local fundamental group of $(\YC(d),0)$ is just the 
fundamental group of $\YC(d)\setminus\{0\}$. 
The above discussion shows that it only remains to check the 
following proposition, whose proof uses in an essential way the 
$\Sb\Lb_2$-action:


\begin{prop}\label{prop:pi1}
The variety $\YC(d)\setminus\{0\}$ is simply-connected.
\end{prop}


\begin{proof}
We fix a base point $p=(0,0,1,1,0,\dots,0,1) \in \YC(d) \setminus\{0\}$ and 
we want to show that $\pi_1(\YC(d) \setminus \{0\},p)=1$. 
We divide the proof into several steps.

\medskip

\noindent{$\bullet$ \it First step: projection to $\sG\lG_2$.} 
%
First, if $(q,Q,e) \in \CM^3$, we denote by $M(q,Q,e)$ the matrix 
$$M(q,Q,e)=\begin{pmatrix} e & 2Q \\ - 2q & -e \end{pmatrix} \in \sG\lG_2.$$
Viewing $\YC(d)$ as a closed subvariety of $\sG\lG_2 \oplus \Sym^d(U_2)$ as 
in Lemma~\ref{orblem2} (we have forgotten the exponents $^{(1)}$ and $^{(-1)}$ 
in $\sG\lG_2 \oplus \Sym^d(U_2)$ because we are only considering the $\Sb\Lb_2$-action, 
and not the $\Gb\Lb_2$-action), the map ${\rm pr} : \YC(d) \to \sG\lG_2$ obtained by 
projecting to the first component is given by 
$$
{\rm pr}(q,Q,e,b_0,b_1,\dots,b_d)=M(q,Q,e).
$$
It is $\Sb\Lb_2$-equivariant. 
Also, the matrices ${\rm pr}(p)=M(0,0,1)=\diag(1,-1)=H$ and $M(1,-1/4,0)$ are in the 
same $\Sb\Lb_2(\CM)$-orbit because they have the same non-zero determinant. We denote 
by $g_0$ an element of $\Sb\Lb_2$ such that $g_0 M(0,0,1) g_0^{-1} = M(1,-1/4,0)$. 

\def\opp{{\mathrm{op}}}
Another easy property of this action is the description of its restriction 
to the diagonal torus of $\Sb\Lb_2$: let $\D^\opp : \CM^\times \longto \Sb\Lb_2$, 
$\xi \mapsto \diag(\xi,\xi^{-1})$. Then 
\equat\label{eq:action-xi}
\D^\opp(\xi) \cdot (q,Q,e,b_0,b_1,\dots,b_d)=(\xi^{-2}q,\xi^2 Q,e,\xi^{-d} b_0, 
\xi^{2-d} b_1,\dots, \xi^{d} b_d).
\endequat

\medskip

\noindent{$\bullet$ \it Second step: fibration.} 
Now, let $\UC$ denote the open subset $(\det \circ {\rm pr})^{-1}(\CM^\times)$ of $\YC(d)$ 
and let $\VC$ denote the open subset $\det^{-1}(\CM^\times)$ of $\sG\lG_2(\CM)$, 
so that $\UC={\rm pr}^{-1}(\VC)$. Then $\VC$ is the open subset of regular semisimple 
elements. Moreover:

\medskip

\begin{quotation}
\noindent{\bf Fact 1.} {\it The map $\det \circ {\rm pr} : \UC \to \CM^\times$ is a principal 
fibration with fiber isomorphic to $\Sb\Lb_2(\CM)/\D^\opp(\mub_d)$.} 

\begin{proof}
Let $\z \in \CM^\times$. Let $\xi$ be a square root of $-\z$ and set $\t=\diag(\xi,-\xi)$, 
so that $\t \in \det^{-1}(\z)$. More precisely, $\det^{-1}(\z)$ is the $\Sb\Lb_2$-orbit 
of $\t$, and the stabilizer of $\t$ in $\Sb\Lb_2$ is the diagonal torus 
$\D^\opp(\CM^\times)$. So it is sufficient to prove that $\D^\opp(\CM^\times)$ 
acts transitively on ${\rm pr}^{-1}(\t)$ and that the stabilizer of some (any) point in 
${\rm pr}^{-1}(\t)$ is equal to $\D^\opp(\mub_d)$. 

Using~(\ref{eq:psi-e}) and the equations for $\YC(d)$, it follows that
\begin{multline*}
\hphantom{AAAAA}{\rm pr}^{-1}(\t)=\{(q,Q,e,b_0,b_1,\dots,b_d) \in \CM^{d+4}~|~
q=Q=b_1=\dots=b_{d-1}=0, \\
e=\xi~\text{and}~ b_0b_d = -\xi^{d-2}\}.\hphantom{AAAAAAAA}
\end{multline*}
Now, the result follows from~\eqref{eq:action-xi}.
\end{proof}
\end{quotation}

\medskip

\noindent{$\bullet$ \it Third step: long exact sequence in homotopy.} 
Applying the long exact sequence in homotopy to Fact~1 yields a short exact sequence 
of groups
\equat\label{eq:exact}
1 \longto \pi_1(\Sb\Lb_2(\CM)/\D^\opp(\mub_d),1) \longto \pi_1(\UC,p) \longto \pi_1(\CM^\times,-1) 
\longto 1.
\endequat
But $\pi_1(\Sb\Lb_2(\CM)/\D^\opp(\mub_d),1) \simeq \mub_d$, generated by the loop 
$\a_0 : t \mapsto \D^\opp(e^{2\sqrt{-1}\pi t/d})\D^\opp(\mub_d)$ while $\pi_1(\CM^\times,-1) \simeq \ZM$, 
generated by the loop $\b_0 : t \mapsto e^{2\sqrt{-1}\pi t}$. Now, let 
$$\fonction{\a}{[0,1]}{\UC}{t}{(0,0,1,e^{-2\sqrt{-1}\pi t},0,\dots,0,-e^{2\sqrt{-1}\pi t})}$$
$$\fonction{\b}{[0,1]}{\UC}{t}{g_0^{-1} \ph(1,0,-e^{2\sqrt{-1}\pi t}/4) g_0.}\leqno{\text{and}}$$
Then $\a$ is the image of $\a_0$ in $\pi_1(\UC,p)$ while $\b$ is a lift 
of $\b_0$ in $\pi_1(\UC,p)$. The short exact sequence of groups~(\ref{eq:exact}) implies that 
\equat\label{eq:chemins}
\pi_1(\UC,p) = \langle \a,\b \rangle.
\endequat

\medskip

\noindent{$\bullet$ \it Last step: conclusion.} 
Recall the following classical fact:

\medskip

\begin{quotation}
\noindent{\bf Fact 2.} {\it The map $\pi_1(\UC,p) \longto \pi_1(\YC(d)\setminus\{0\},p)$ 
is surjective.} 

\begin{proof}
By Proposition~\ref{prop:blowup}, the variety $\YC(d) \setminus \{0\}$ is smooth, 
and $\UC$ is a Zariski open subset of $\YC(d) \setminus \{0\}$, so the 
result follows from~\cite[Theorem~2.3]{godbillon}.
\end{proof}
\end{quotation}

\medskip

\noindent Fact~2 shows that it suffices to 
check that $\a$ and $\b$ are homotopy equivalent, in $\YC(d) \setminus \{0\}$, 
to the trivial loop. But the loops $\a$ and $\b$ are both contained in 
$g_0^{-1} \ph(\SC) g_0$, 
which is simply-connected by Lemma~\ref{lem:pi1-x}. So the proof 
of the proposition is complete.
\end{proof}


Together with Theorem~\ref{theo:zd}, this concludes the proof of Theorem~\ref{theo:main 1}. 


\section{Quiver varieties and Slodowy slices}

\medskip

In this section we show that the isolated singularity $(\ZC(d),0)$ is locally analytically isomorphic 
to a certain singularity $(\XCt(d),\xti_d)$ constructed as a cover of a type $A$ Slodowy slice. The isomorphism passes through a Nakajima quiver variety. 

We begin with a quiver of type $\widetilde{A}$ with $d$ vertices, labelled $\rho_0, \dots, \rho_{d-1}$ and $d$ arrows $\rho_i \to \rho_{i+1}$, where indices are taken modulo $d$. We extend this by adding a framing $\rho_{\infty} \to \rho_0$ and denote by $\Qu$ the resulting quiver. Let $\overline{\Qu}$ be the doubled quiver, which is independent of the choice of orientation of $\Qu$. 

We let $\delta$ be the minimal imaginary root for the 
$\widetilde{A}$ root system, thought of as a dimension vector for $\overline{\Qu}$, and set $\bv = \rho_{\infty} + 2 \delta$. Let $Q = \bigoplus_{i = 1}^{d-1} \ZM \rho_i$ be the root lattice of the root system $\Phi$ of type $A_{d-1}$, with set of simple roots $\{ \rho_1, \dots, \rho_{d-1} \}$. Then $Q^+ := \bigoplus_{i = 1}^{d-1} \ZM_{\ge 0} \rho_i$ and the set of positive roots is $\Phi^+ = Q^+ \cap \Phi$. Let $\alpha_h = \rho_1 + \cdots + \rho_{d-1}$ be the highest positive root. For $1 \le i < j \le d-1$, we define $\alpha_{i,j} = \rho_i + \cdots + \rho_j$, an element of $\Phi^+$. The symbol $\prec$ will refer to the dominance ordering on $Q$. We have
\equat\label{eq:phinorm2}
\{ \alpha \in Q^+ \, | \, (\alpha,\alpha) = 2 \} = \Phi^+ .
\endequat

\begin{lem}
The following holds: 
\equat\label{eq:Qlength4A}
\{ \alpha \in Q^+ \, | \, \textrm{$(\alpha,\alpha) = 4$, and $\alpha \preceq 2 \alpha_h$} \} = \{ \alpha_{i,j} + \alpha_{k,l} \, | \, 1 \le i < k < l < j \le d-1 \} \\
\endequat
\equat\label{eq:Qlength4B}
\hphantom{\{ \alpha \in Q^+ \, | \, \textrm{$(\alpha,\alpha) = 4$, and $\alpha \preceq 2 \alpha_h$} \} }
\cup \{ \alpha_{i,j} + \alpha_{k,l} \, | \, 1 \le i < j < k < l \le d-1 \}. 
\endequat

\end{lem}

\begin{proof}
	
	When $1 \le i < k < l < j \le d-1$ (where the "shorter'' root $\alpha_{k,l}$ sits on top of the "higher'' root $\alpha_{i,j}$) or when $ 1 \le i < j < k < l \le d-1$ (where we have a disjoint union of two roots) we have $(\alpha_{i,j}, \alpha_{k,l}) = 0$ and hence 
	$$
	(\alpha_{i,j} + \alpha_{k,l}, \alpha_{i,j} + \alpha_{k,l}) = (\alpha_{i,j}, \alpha_{i,j}) + (\alpha_{k,l},\alpha_{k,l}) = 2 + 2 = 4. 
	$$  
	Therefore, the right hand sides of~\eqref{eq:Qlength4A} 
	and~\eqref{eq:Qlength4B} are contained in the left hand side. 
	
	Assume now we have $\alpha \in Q^+$ with $(\alpha,\alpha) = 4$. If the support of $\alpha$ is disconnected, we can write $\alpha = \alpha^{(1)} + \alpha^{(2)}$ with the 
	$\alpha^{(i)} \neq 0$ having orthogonal supports. Then $4 = (\alpha,\alpha) = (\alpha^{(1)},\alpha^{(1)}) + (\alpha^{(2)},\alpha^{(2)})$. Since $(\beta,\beta) \ge 2$ for all non-zero $\beta \in Q^+$, we deduce from \eqref{eq:phinorm2} that $\alpha^{(i)} \in \Phi^+$ and hence $\alpha$ belongs to the set in \eqref{eq:Qlength4B}. 
	
	Thus, we assume that the support of $\alpha$ is connected. As a (connected) subgraph of a type $A$ Dynkin diagram is again a type $A$ Dynkin diagram, we may assume without loss of generality that the support of $\alpha$ is the whole Dynkin diagram. Thus, $\alpha_h \preceq \alpha \preceq 2 \alpha_h$ and we can write 
	$$
	\alpha = 2 \alpha_h - (\alpha_{i_1,j_1} + \cdots + \alpha_{i_k,j_k})
	$$
	where $1 \le i_1 \le j_1 < i_2 \le \cdots < i_k \le j_k \le d-1$ and moreover $j_m+1 < i_{m+1}$. 
	We compute
	\begin{align*}
		(\alpha,\alpha) & = 4 (\alpha_h,\alpha_h) + \sum_{m=1}^k (\alpha_{i_m,j_m}, \alpha_{i_m,j_m}) - 4 \sum_{m=1}^k (\alpha_h,\alpha_{i_m,j_m}) \\
		& = 8 + 2k - 4(\delta_{i_1,1} + \delta_{j_k,d-1}). 
	\end{align*}
	If this equals $4$ then we must have $k = 2$ and $i_1 = 1, j_2 = d-1$. Hence $\alpha = \alpha_{1,d-1} + \alpha_{j_1 + 1,i_2-1}$ belongs to the right hand side of \eqref{eq:Qlength4A}. 
\end{proof}


Now we consider the parameter $\lambda$, where $\lambda_{\infty} = -2, \lambda_0 = 1$ and $\lambda_i = 0$ otherwise, for the quiver variety associated to $\Qu$. Notice that $\lambda \cdot \bv = \lambda \cdot ( \rho_{\infty} + 2 \delta) = 0$.  Therefore we can associate to it the affine quiver variety $\MG_{\lambda}(\bv)$ as defined in~\cite[Section~1.1]{BellSchedQuiver}. This space parameterises semi-simple representations of the deformed preprojective algebra $\Pi^{\lambda}(\Qu)$ of dimension $\bv$. We define the set $\Sigma_{\lambda}(\bv)$ to be all positive roots $\alpha$ for $\Qu$ such that 
\begin{itemize}
	\itemth{A} $\alpha \preceq \bv$ and $\lambda \cdot \alpha = 0$, 
	\itemth{B} $p(\alpha) > p(\beta^{(1)}) + \cdots + p(\beta^{(k)})$ for all proper decompositions $\alpha = \beta^{(1)} + \cdots + \beta^{(k)}$ with $\beta^{(i)}$ a positive root, $\lambda \cdot \beta^{(i)} = 0$.
\end{itemize}    
Here $p(\alpha) := 1 - (1/2)(\alpha,\alpha)$. A vector $\alpha$ belongs to $\Sigma_{\lambda}(\bv)$ if and only if there is a simple representation of $\Pi^{\lambda}(\Qu)$ of dimension $\alpha \le \bv$. 


\begin{lem}\label{lem:Sigmabvlambda}
$\Sigma_{\lambda}(\bv) = \{ \bv, \rho_{\infty} + 2 \rho_0 + \alpha_h, \rho_{\infty} + 2 \rho_0 + \rho_1 + \rho_{d-1}, \rho_1, \dots, \rho_{d-1} \}$.		
\end{lem}


\begin{proof}
	If $\alpha \in \Sigma_{\lambda}(\bv)$ then either $\alpha_{\infty} = 1$ or $\alpha_{\infty} = 0$. Consider first the latter. Since $\lambda \cdot \alpha = 0$, this implies that $\alpha_0 = 0$ too. That is, $\alpha \in \Phi^+$. But $\Sigma_{\lambda}(\bv) \cap \Phi^+ = \{ \rho_1, \dots, \rho_{d-1} \}$. 
	
	Therefore, we assume that $\alpha_{\infty} = 1$. As noted in~\cite[Lemma~4.3]{BellamyCrawQuotient}, this implies that $\alpha = \rho_{\infty} + m \delta + (1/2)(v,v) \delta - v$, for some $v \in Q$ and $m \in \ZM_{\ge 0}$. Since $\lambda \cdot \alpha = 0$, we must have $m + (1/2) (v,v) = 2$. Moreover, for $i = 1, \dots, d-1$, $m + (1/2)(v,v) - v_i \le 2$ because $\delta_i = 1$. We deduce that $v \in Q^+$. 
	
	Next, for all $\alpha' = \rho_{\infty} + m' \delta + (1/2)(v',v') \delta - v'$ with $\alpha' \prec \alpha$ and $m' + (1/2) (v',v') = 2$, we require $p(\alpha) > p(\alpha')$ because $\alpha - \alpha'$ is a sum of real roots. If $\alpha > \alpha'$ then $v' \succ v$. Moreover, $2 - (1/2) (v, v) = p(\alpha) > p(\alpha') = 2 - (1/2) (v', v')$ if and only if $(v', v') > (v, v)$. Thus, we require: 
	\begin{center}
		$(v', v') > (v, v)$ for all $v' \in Q^+, v' \succ v$ and $(v', v') \le 4$. 
	\end{center}
	First, if $v = 0$ then $(v,v) = 0$ so the condition is vacuous. Next, we assume that $(v,v) = 2$. Then $v \in \Phi^+$ by~\eqref{eq:phinorm2} and the condition says $v$ is maximal with respect to $\preceq$ on $\Phi^+$. In other words, $v = \alpha_h$ is the highest root. This corresponds to $\alpha = \rho_{\infty} + 2 \rho_0 + \alpha_h$. Finally, if $(v,v) = 4$ then the condition says that $v$ should be a maximal vector in the union of \eqref{eq:Qlength4A} and \eqref{eq:Qlength4B}. There is only one maximal vector, which is $\alpha_{1,d-1} + \alpha_{2,d-2}$. This corresponds to $\alpha= \rho_{\infty} + 2 \rho_0 + \rho_1 + \rho_{d-1}$. 
\end{proof}


\begin{lem}
	The quiver variety $\MG_{\lambda}(\bv)$ has three symplectic leaves, $\MG_{\lambda}(\bv)_{\tau_0},\MG_{\lambda}(\bv)_{\tau_2}$ and $\MG_{\lambda}(\bv)_{\tau_4}$, of dimension $0$, $2$ and $4$ respectively. 
\end{lem}


\begin{proof}
	As shown in \cite{BellSchedQuiver}, the sympectic leaves of $\MG_{\lambda}(\bv)$ are labelled by the representation types $\tau = (\beta^{(1)},n_1;\dots ; \beta^{(k)},n_k)$ of $(\bv,\lambda)$. Here $\beta^{(i)} \in \Sigma_{\lambda}(\bv)$, $\bv = n_1 \beta^{(1)} + \cdots + n_k \beta^{(k)}$ and the real roots in $\Sigma_{\lambda}(\bv)$ occur at most once amongst the $\beta^{(i)}$. The leaf labelled by $\tau$ has dimension $\sum_i 2 p(\beta^{(i)})$. In our case, Lemma~\ref{lem:Sigmabvlambda} implies that  the possible representation types are
	\begin{align*}
		\tau_0 & = (\rho_{\infty} + 2 \rho_0 + \rho_1 + \rho_{d-1},1;\rho_1,1;\rho_2,2; \dots ; \rho_{d-1},1) \\
		\tau_2 & = (\rho_{\infty} + 2 \rho_0 + \alpha_h,1;\rho_1,1;\rho_2,1; \dots ; \rho_{d-1},1) \\
		\tau_4 & = (\bv,1).
	\end{align*}
\end{proof}


Since the symplectic leaf $\MG_{\lambda}(\bv)_{\tau_0}$ is zero-dimensional and connected, it equals $\{ x_0 \}$ for some point $x_0 \in \MG_{\lambda}(\bv)$. 

Recall from the introduction that $\mathcal{S}_{d-2,2}$ is a Slodowy slice associated with the nilpotent orbit $\OC_{d-2,2}$, that $\XC(d) = \mathcal{S}_{d-2,2} \cap \NC_d$, and that $x_d$ is the unique element of 
$\OC_{d-2,2} \cap \SC_{d-2,2}$. 


\begin{theo}\label{thm:quiverslodlwyiso}
	The singularities $(\MG_{\lambda}(\bv),x_0)$ and $(\XC(d),x_d)$ are locally analytically isomorphic. 
\end{theo}


\begin{proof}
	This is an application of Crawley-Boevey's \'etale local picture \cite{CBnormal}, together with the isomorphism of Nakajima \cite{Nak1994} (see also~\cite{Ma}). The point $x_0 \in \MG_{\lambda}(\bv)_{\tau_0}$ corresponds to a semi-simple representation of the deformed preprojective algebra $\Pi^{\lambda}(\Qu)$ of the form
	$$
	M = M_{\infty} \oplus M_1 \oplus M_2^{\oplus 2} \oplus \cdots \oplus M_{d-2}^{\oplus 2} \oplus M_{d-1},
	$$
	where all summands are simple, $\dim M_{\infty} = \rho_{\infty} + 2 \rho_0 + \rho_1 + \rho_{d-1}$ and $\dim M_i = \rho_i$. Since the dimension vector of each simple summand is real, there is a unique (up to isomorphism) simple representation of that dimension. 
	
	Corollary 4.10 of \cite{CBnormal} says that $(\MG_{\lambda}(\bv),x_0)$ is (\'etale locally) isomorphic to $0$ in another (framed) quiver variety; as shown in~\cite[Theorem~3.3]{BellSchedQuiver} this isomorphism is Poisson. The vertices $e_i$ of the new quiver are in bijection with the $M_i$, and $M_{\infty}$ corresponds to the framing data. There are $-(\dim M_i,\dim M_j) = \delta_{i,j-1} + \delta_{i,j+1}$ arrows from $e_i$ to $e_j$ in the doubled quiver. That is, the new (undoubled) quiver has underlying graph Dynkin of type $A_{d-1}$. Since all vectors are real there are no loops at any vertices. Finally there are 
	$$
	-(\dim M_{\infty},\dim M_i) = \left\{ \begin{array}{ll}
		1 & i = 2, d-2, \\
		0 & \textrm{otherwise,}
	\end{array} \right.
	$$   
	arrows from the framing vertex $e_{\infty}$ to $e_i$. As a (framed) quiver variety,
	$$
	\begin{tikzcd}
	w : & 	 & 1 \ar[d,dash] & & & & 1 \ar[d,dash] & \\
	v : & 1 \ar[r,dash] & 2 \ar[r,dash] & 2 \ar[rr,dash,dotted] & & 2 \ar[r,dash] & 2 \ar[r,dash] & 1  	 
	\end{tikzcd}
	$$
	Here $w$ is given by $w_2 = w_{d-2} = 1$ and $w_i = 0$ otherwise and $v_1 = v_{d-1} = 1$ and $v_2 = 2$ otherwise. Following the construction given in section 8 of \cite{Nak1994}, we see that $\boldsymbol{\mu} = (d)$ and $\boldsymbol{\lambda} = (d-2,2)$. 
	Therefore, the isomorphism of~\cite[Remark~8.5(2)]{Nak1994} says that the above framed quiver variety is isomorphic to $\XC(d)$. It follows from~\cite[Lemma~4.6.4]{Losev} that this isomorphism is as Poisson varieties. 
\end{proof}


\begin{rema}
	If $M$ is a representation of $\Pi^{\lambda}$ lying on the $2$-dimensional symplectic leaf $\MG_{\lambda}(\bv)_{\tau_2}$ then 
	$$
	M = M_{\infty} \oplus M_1 \oplus M_2 \oplus \cdots \oplus M_{d-1} \oplus M_{d-1},
	$$
	with $\dim M_i = \rho_i$. The simple summands correspond to vertices $e_{\infty}, e_1,\dots, e_{d-1}$. The associated doubled quiver is the affine type $A$ quiver, except now $p(\rho_{\infty} + 2 \rho_0 + \alpha_h) = 1$, which implies that there are two loops at $e_{\infty}$. The dimension vector is the minimal imaginary root $\delta = (1,\dots, 1)$. This implies that $(\MG_{\lambda}(\bv), [M])$ is isomorphic to $(\CM^2 \times (\CM^2 / \mub_d), 0)$. 
\end{rema}


In the notation of~\cite[Section~7.3]{MarsdenWeinsteinStratification}, we choose $c_1$ generic, $\underline{c} = 0$, and let $\ZC(d,1,2)$ denote the spectrum of the centre of the rational Cherednik algebra associated to the wreath product group $G(d,1,2)=\mub_d\wr\mathfrak{S}_2 $ at $t = 0$ and $(c_1,\underline{c})$. By~\cite[Theorem~1.4]{MarsdenWeinsteinStratification}, we have an isomorphism $\ZC(d,1,2) \cong \MG_{\lambda}(\bv)$, which is Poisson up to a scalar factor. Therefore, Theorem~\ref{thm:quiverslodlwyiso} can be reinterpreted as saying that there is an isomorphism of symplectic singularities $(\ZC(d,1,2),x_0) \cong (\XC(d),x_d)$. Recall that the nilpotent cone $\NC_d$ of 
$\sG\lG_d(\CM)$ admits a $\mub_d$-covering $\pi_d : \NCt_d \longto \NC_d$
which is unramified above the regular nilpotent orbit $\OC_d$ and bijective above the branch locus. 
We have set $\XCt(d)=\pi_d^{-1}(\XC(d))$ and we denote by $\xti_d$ the unique element 
in $\pi_d^{-1}(x_d)$. We are now ready to prove our second main result:

\def\smooth{{\mathrm{sm}}}


\begin{theo}
If $d \ge 4$, the symplectic singularities $(\ZC(d),0)$ and $(\XCt(d),x_d)$ are locally analytically isomorphic.
\end{theo}


\begin{proof}
	The group $W_d$ is a normal subgroup of $G(d,1,2)$ with quotient $\mub_d$. 
	By~\cite[Proposition~4.17]{BellThielCusp}, the group $\mub_d$ acts on $\ZC(d)$ such that $\ZC(d) / \mub_d \cong  \ZC(d,1,2)$. From now on, we consider all spaces as complex analytic spaces.
 
Theorem~\ref{thm:quiverslodlwyiso} says that, in the analytic topology, there exist (analytic) open balls $B$ and $D$ around $x_0$ and $x_d$ in $\ZC(d,1,2)$ and $\XC(d)$ respectively such that $\varphi \colon D \stackrel{\sim}{\longrightarrow} B$, sending $x_d$ to $x_0$. Let $\widetilde{B}$ and $\widetilde{D}$ denote the preimages of $B$ and $D$ in $\ZC(d)$ and $\XCt(d)$ respectively. By~\cite[Theorem~4]{ku}, $\widetilde{B}$ and $\widetilde{D}$ are normal complex spaces. 

The preimage of the smooth locus of $B$ and $D$ in $\widetilde{B}$ and $\widetilde{D}$ are denoted by 
$\widetilde{B}^{\circ}$ and $\widetilde{D}^{\circ}$ respectively. The complement to these open sets has (complex) codimension two. Thus Proposition~\ref{prop:pi1} says that $\pi_1(\widetilde{B}^{\circ}) = 1$. Hence $\pi_1(B_{\mathrm{sm}}) = \mub_d$ since the map $\widetilde{B}^{\circ} \to B_{\mathrm{sm}}$ is 
unramified and Galois with group $\mub_d$. Therefore, there exists a (necessarily holomorphic; see \cite{Extendholomorphi}) covering map $\psi^{\circ} \colon \widetilde{B}^{\circ} \to \widetilde{D}^{\circ}$ 
making the diagram 
$$
\begin{tikzcd}
	\widetilde{B}^{\circ} \ar[r,"\psi^{\circ}"] \ar[d] & \widetilde{D}^{\circ} \ar[d] \\
	 B_{\mathrm{sm}} &  D_{\mathrm{sm}} \ar[l,"\varphi"']
\end{tikzcd}
$$
commutative. Since the degree of both vertical maps is $d$, we deduce that $\psi^{\circ}$ is an isomorphism. Composing $\psi^{\circ}$ with the embedding $\widetilde{D}^{\circ} \hookrightarrow \XCt(d)$ gives a map $\psi \colon \widetilde{B}^{\circ} \to \XCt(d)$. 
By~\cite[VI,~Proposition~3.1]{HartshorneAmplesub}, the complex space $\XCt(d)$ is Stein. 
Therefore,~\cite[Theorem~2]{Extendholomorphi} says that $\psi$ extends to a holomorphic map $\psi \colon \widetilde{B} \to \XCt(d)$. 

Consider next the diagram 
$$
\begin{tikzcd}
	\widetilde{B} \ar[rr,"\psi"] \ar[d] & & \XCt(d) \ar[d,"\pi_d"] \\
	B \ar[r,"\varphi^{-1}"] &  D \ar[r,hook] & \XC(d).
\end{tikzcd}
$$
It is commutative when restricted to the dense open set $\widetilde{B}^{\circ}$. Therefore, it is everywhere commutative. Since the composition $\widetilde{B} \to B \to D \hookrightarrow \XC(d)$ has image equal to $D$, we deduce that the image of $\pi_d \circ \psi$ equals $D$. The map $\psi$ is $\mub_d$-equivariant (since it is so generically). Therefore, the fact that $\widetilde{D} = \pi_d^{-1}(D)$ means that the image of $\psi$ equals $\widetilde{D}$ and $\psi$ is bijective. Finally, since both $\widetilde{B}$ and $\widetilde{D}$ are normal, the analytic version of Zariski's main theorem implies that $\psi$ is biholomophic. In particular, 
the singularities $(\ZC(d),0)$ and $(\XCt(d),\xti_d)$ are locally analytically isomorphic.     
\end{proof}



\begin{coro}\label{coro:xd-isstlfg}
The isolated symplectic singularity $(\XCt(d),\xti_d)$ has trivial local fundamental group.
\end{coro}


\section{Complements}


\subsection{Further deformation of ${\boldsymbol{\ZC(d)}}$} 
If $d \ge 4$ is even, then the Poisson variety $\ZC(d)$ admits a further
deformation~\cite{EG}, which is still a Calogero-Moser space associated with the
dihedral group $W_d$ (see~\cite{bonnafe diedral}). Let $\ZC^\#(d)$
be a generic such deformation. The variety $\ZC^\#(4)$ is smooth so is
uninteresting for our purpose but, if $d \ge 6$ then $\ZC^\#(d)$
admits a single singular point that we still denote by $0$ (see~\cite{bellamy these}
or~\cite{bonnafe diedral}).

So assume from now on that $d \ge 6$ is even. Again, general facts about Calogero-Moser spaces
say that $(\ZC^\#(d),0)$ is a symplectic singularity~\cite[Proposition~4.5]{gordon icra} and we denote by
$\Trm_0^*(\ZC^\#(d))$ the cotangent space of $\ZC^\#(d)$ at $0$, endowed
with its Lie algebra structure induced by the Poisson bracket on $\ZC^\#(d)$.
The smallest singular case $d=6$ is somewhat particular:


\begin{prop}\label{prop:sp4}
The singularities $(\ZC^\#(6),0)$ and $(\overline{\OC}_\mini^{\sG\pG_4},0)$
are locally analytically isomorphic.
\end{prop}


\begin{proof}
By~\cite[Proposition~8.8]{bonnafe diedral}, we have an isomorphism of Lie algebras
$\Trm_0^*(\ZC^\#(6)) \simeq \sG\pG_4$ and computations with {\sc Magma}~\cite{magma},~\cite{thiel}, 
based on the equations of $\ZC^\#(6)$ given in~\cite[Table~5]{bonnafe diedral}
show that the projective tangent cone of $\ZC^\#(6)$ at $0$ is smooth.
So the result follows from Beauville's Theorem~\ref{theo:beauville}.
\end{proof}


In other words, the singularity $(\ZC^\#(6),0)$ is locally analytically isomorphic to the singularity
$(\CM^4/\mub_2,0)$. This shows in particular that the local fundamental group
of $(\ZC^\#(6),0)$ is isomorphic to $\mub_2$. 

In the general even case, the action of $\mub_2$ on $\ZC(d)$ given by 
$$(q,Q,e,a_0,\ldots ,a_m)\longmapsto (q,Q,e,-a_0,\ldots ,-a_d)$$
is free on the smooth part 
of an open neighbourhood of $0$, hence $\ZC(d)/\mub_2$ has an isolated singularity at $0$.
By the argument in the proof of Prop. \ref{prop:blowup}, this is locally analytically isomorphic 
to $(\YC(d)/\mub_2,0)$ (with the same action on $\CM^{d+4}$).
It is not difficult to show that the Calogero-Moser space associated with $W_{2d}$ and (non-generic) 
parameters $(c_1,c_2)=(0,1)$ is isomorphic to $\ZC(d)/\mub_2$.


\begin{quotation}
\noindent{\bf Question:} {\it Assume $d\geq 4$, is the singularity 
$(\ZC^\#(2d),0)$ locally analytically isomorphic to 
$(\ZC(d)/\mub_2,0)$?}
\end{quotation}


We note here a further description of $\YC(d)/\mub_2$.
Recall by Corollary \ref{coro:singlocus} that the singular locus of $\QC(2d)$ has two irreducible components 
$C_1$ and $C_2$, with ideals $J_1$ and $J_2$.
We can cover the blowup of $\QC(2d)$ at $C_1$ by affine open subsets as as we did in \S 2.C.
It is easy to see that the first affine open subset $\Spec(\CM[V\times V^*]^{W_{2d}}[J_1/D])$ 
is isomorphic to $\YC(d)/\mub_2$.
In fact, the singular locus of this blowup is a disjoint union of a $2$-dimensional subset 
(the pre-image of $C_2$, which can be resolved by blowing up once more) and the singular point 
$0\in\YC(d)/\mub_2$. Similarly, $\Spec(\CM[V\times V^*]^{W_{2d}}[J_2/D]$ is isomorphic to $\YC(d)/\mub_2$.

%
%
%
%


\subsection{Hilbert series}\label{subsec:hilb}
We compute here the Hilbert series of the graded algebra $\CM[\YC(d)]$. 
For this, recall that any symplectic singularity is Cohen-Macaulay 
(and even Gorenstein~\cite[Proposition~1.3]{beauville}). 
Therefore, $\CM[\YC(d)]$ is a graded Cohen-Macaulay ring. 
Let $R=\CM[q,Q,b_0,b_d]$ and let $\mG$ denote the unique 
graded maximal ideal of $R$. If $f \in \CM[\YC(d)]$, we denote by 
$\fba$ its image in $\CM[\YC(d)]/\langle \mG \rangle$. Then
\equat\label{eq:somme}
\CM[\YC(d)]/\langle \mG \rangle=\CM \oplus \CM \eba \oplus \cdots \oplus \CM \eba^{d-2} 
\oplus \CM \bba_1 \oplus \CM \bba_2 \oplus \cdots \oplus \CM \bba_{d-1}
\endequat
\begin{proof}[Proof of~\eqref{eq:somme}]
By definition, $\CM[\YC(d)]/\langle \mG \rangle$ is the commutative 
$\CM$-algebra whose presentation is given by:
$$
\begin{cases}
\text{Generators:}\quad \eba,\bba_1,\dots,\bba_{d-1} \\
\text{Relations:} \quad (\#)\quad
\begin{cases}
\forall 1 \le j \le d-1,~\eba \bba_j =0\\
\forall 1 \le j \le k \le d-1,~\bba_j \bba_k =
\begin{cases}
0 & \text{if $j+k \neq d$,}\\
-\eba^{d-2} & \text{if $j+k=d$.}\\
\end{cases}
\end{cases}
\end{cases}
$$
Then $\CM + \CM \eba + \cdots + \CM \eba^{d-2} + 
\CM \bba_1 + \CM \bba_2 + \cdots + \CM \bba_{d-1}$ is a subalgebra 
of $\CM[\YC(d)]/\langle \mG \rangle$ (because $\eba^{d-1}=\eba \bba_1\bba_{d-1}=0$), 
which contains all the generators. So it is equal to $\CM[\YC(d)]/\langle \mG \rangle$. 
It only remains to show that $\dim \CM[\YC(d)]/\langle \mG \rangle \ge 2d-2$. 

For this, let $E_{k,l}$ denote the elementary $(2d-2) \times (2d-2)$-matrix 
whose only non-zero entry is is the $(k,l)$-entry, which is equal to $1$. We 
set
$$\Eba=E_{2,1}+E_{3,2}+\cdots+E_{d-1,d-2}\qquad\text{and}\qquad 
\Bba_j=E_{d-1+j,1}-E_{d-1,2d-j-1}$$
for $1 \le j \le d-1$. Then the relations~(\#) show that 
there is a unique morphism of algebras 
$\th : \CM[\YC(d)]/\langle \mG \rangle \longto \Mat_{2d-2}(\CM)$ sending $\eba$ to $\Eba$ 
and $\bba_j$ to $\Bba_j$. As 
$\th(1)$, $\th(\eba)$,\dots, $\th(\eba^{d-2})$, $\th(\bba_1)$,\dots, $\th(\bba_{d-1})$ 
are linearly independent, this proves the result.
\end{proof}


By~\eqref{eq:somme}, $\CM[\YC(d)]/\langle \mG \rangle$ is finite-dimensional, 
so the graded Nakayama Lemma implies that $\CM[\YC(d)]$ is a finitely 
generated $R$-module. This shows that $(q,Q,b_0,b_d)$ is a system of 
parameters: since $\CM[\YC(d)]$ is Cohen-Macaulay, this shows that 
\equat\label{eq:somme directe}
\CM[\YC(d)]=R \oplus R e \oplus \cdots \oplus R e^{d-2} 
\oplus R b_1 \oplus R b_2 \oplus \cdots \oplus R b_{d-1}.
\endequat
The Hilbert series $\Hrm_{\YC(d)}(\tb)$ of $\CM[\YC(d)]$ is then easily computed:
\equat\label{eq:hilbert series}
\Hrm_{\YC(d)}(\tb)=
\cfrac{1 + \tb^2 + \cdots + \tb^{2d-4} + (d-1) \tb^{d-2}}{(1 - \tb^2)^2 (1 - \tb^{d-2})^2}.
\endequat


\subsection{Higher dimension} 
We explain here why the $\Sb\Lb_2$-equivariant morphism 
$\QC(d) \to \sG\lG_2$ is a particular case of a more general situation. 
For this, let $(O,(.,.))$ be a finite dimensional orthogonal vector space 
and let $(S,\o)$ be a finite dimensional symplectic vector space. 
For $f \in \Hom(S,O)$ (resp. $g \in \Hom(O,S)$), we denote by $f^*$ (resp. $g^*$) 
the unique element of $\Hom(O,S)$ (resp. $\Hom(S,O)$) such that 
$(f(s),o)=\o(s,f^*(o))$ (resp. $\o(g(o),s)=(o,g^*(s))$) for all $o \in O$ and $s \in S$. 
Note that $f^{**}=-f$ and $g^{**}=-g$. Moreover, the morphisms
$$
\xymatrix{
& \Hom(S,O) \ar[dl]_{\pi_{\sG\pG}} \ar[dr]^{\pi_{\sG\oG}} \\
\sG\pG(S) && \sG\oG(O)
}
$$
where $\pi_{\sG\pG}(f) = f^*f$ and $\pi_{\sG\oG}(f) = ff^*$ are the quotient maps by $\GO(O)$ 
and $\Sp(S)$ respectively, which are equivariant with respect to $\Sp(S)$ and $\GO(O)$ respectively \cite{KP}.


\begin{exemple}\label{ex:sosp}
Assume here that $S=U_2$, the natural module of $\Gb\Lb_2$ endowed with the symplectic 
form given by the matrix $\begin{pmatrix} 0 & -1 \\ 1 & 0 \end{pmatrix}$, and that 
$O=V$ endowed with the symmetric bilinear 
form given by the matrix $\begin{pmatrix} 0 & 1 \\ 1 & 0 \end{pmatrix}$. 
Note that the quadratic form associated with $(.,.)$ is $Q \in \CM[V]$. 
Then 
$$V \times V^* \simeq V \times V \simeq U_2 \otimes V \simeq U_2^* \otimes V \simeq \Hom(U_2,V) 
\simeq \Mat_2(\CM)$$
(see Remark~\ref{rem:explicit} for the first two isomorphisms). Through 
this sequence of isomorphisms, the action of $\Ob(V)$ (resp. $\Gb\Lb_2$) 
is by left (resp. right) multiplication. Note that 
$$\sG\pG(S)=\sG\lG_2\qquad\text{and}\qquad \sG\oG(V)=\{\diag(z,-z)~|~z \in \CM\}.$$
To respect the notation of the paper, 
if $f \in \Mat_2(\CM)$ has matrix $\begin{pmatrix} x & Y \\ y & X \end{pmatrix}$, 
then $f^*$ has matrix 
$\begin{pmatrix} 0 & -1 \\ 1 & 0 \end{pmatrix}\begin{pmatrix} x & Y \\ y & X \end{pmatrix}
\begin{pmatrix} 0 & 1 \\ 1 & 0 \end{pmatrix}=\begin{pmatrix} X & Y \\ -y & -x \end{pmatrix}$.
Computing $\pi_{\sG\pG}(f) = f^{*}f $ and $\pi_{\sG\oG}(f) = ff^{*}$, we get the following 
commutative diagrams 




$$
\xymatrix{
& \hskip-1cm V \times V^* \simeq \Mat_{2}(\CM) \hskip-1cm
\ar[dl]_{\pi_{\sG\pG}} \ar[dr]^{\pi_{\sG\oG}}&\\
\sG\pG(U_2) \simeq \sG\lG_{2}\hskip-1cm \ar[dr]^{\det}&& \hskip-1cm\sG\oG(V) \simeq \CM \ar[dl]_{\det}\\
& \CM & 
}
\qquad 
\xymatrix{
& {\left( \begin{smallmatrix} x & Y \\ y & X \end{smallmatrix} \right)} \xymap[dl]_{\pi_{\sG\pG}} \xymap[dr]^{\pi_{\sG\oG}} &\\
{\left( \begin{smallmatrix} xX + yY & 2XY \\ -2xy & - xX - yY \end{smallmatrix} \right)}
\hskip-1cm \xymap[dr]^\det
&& \hskip-1cm 
{\left( \begin{smallmatrix} xX - yY & 0 \\ 0 & - xX + yY \end{smallmatrix} \right)}\xymap[dl]_\det \\
& \hskip-1cm \substack{\vphantom{A} \\ \DS{-(xX - yY)^2}}  \hskip-1cm &
}
$$
The map $\pi_{\sG\pG}$ factorizes through the quotient by $\Ob(V)$ but, since 
$W_d \subset \Ob(V)$, it factorizes through a map $\QC(d) \longto \sG\lG_2$ 
which is nothing but the composition of the first projection with the closed immersion 
$\QC(d) \injto \sG\lG_2 \oplus \Sym^d(U_2)$. 
On the other hand, $\pi_{\sG\oG}$ identifies with the map $\d$ 
(once we identify $\sG\oG(V)$ with $\CM$). 
The relation $\det f^*f = \det ff^*$ gives $\delta^2 = e^2 - 4qQ$ 
(which can also be checked directly, of course).\finl
\end{exemple}


It is a natural question whether the construction of Example~\ref{ex:sosp} 
can be extended to higher dimensional settings to produce 
other interesting symplectic singularities.


\end{document}